\documentclass[12pt]{amsart}
\usepackage[utf8]{inputenc}
\usepackage{amssymb,hyperref,mathrsfs, mathtools,cancel}
\usepackage{hyperref}
\usepackage[weather]{ifsym}
\usepackage[lmargin=1in,rmargin=1in,tmargin=1in,bmargin=1in]{geometry}
\usepackage{enumitem}
\usepackage{color}
\usepackage{tikz}

\newcommand{\co}{\nobreak\mskip2mu\mathpunct{}\nonscript
  \mkern-\thinmuskip{:}\penalty300\mskip6muplus1mu\relax}
\renewcommand{\th}{^{\text{th}}}
\newcommand{\OneHalf}{{\textstyle\frac{1}{2}}}
\newcommand{\lsub}[2]{{}_{#1}#2}
\newcommand{\lsup}[2]{{}^{#1}\mskip-.6\thinmuskip#2}

\def\mathcenter#1{%
  \vcenter{\hbox{$#1$}}%
}

\newcommand{\ZZ}{\mathbb{Z}}
\newcommand{\RR}{\mathbb{R}}
\newcommand{\QQ}{\mathbb{Q}}

\newcommand{\Field}{\mathbf{k}}
\newcommand{\Ftwo}{\mathbb{F}_2}
\newcommand{\op}{\mathrm{op}}
\newcommand{\Ainf}{A_\infty}

\newcommand{\bdy}{\partial}

\newcommand{\dg}{\textit{dg }}
\newcommand{\Id}{\mathit{Id}}

\DeclareMathOperator{\Mor}{Mor}
\DeclareMathOperator{\Hom}{Hom}

\DeclareMathOperator{\rank}{rank}
\newcommand{\into}{\hookrightarrow}

\newcommand{\bD}{\mathbb{D}}

\newcommand{\HF}{\mathit{HF}}
\newcommand{\HFa}{\widehat{\mathit{HF}}}
\newcommand{\CF}{{\mathit{CF}}}
\newcommand{\CFa}{\widehat{\mathit{CF}}}
\newcommand{\CFK}{\mathit{CFK}}
\newcommand{\CFKa}{\widehat{\CFK}}
\newcommand{\tHF}{\underline{\HF}}
\newcommand{\tHFa}{\widehat{\tHF}}
\newcommand{\tCF}{\underline{\CF}}
\newcommand{\tCFa}{\widehat{\tCF}}

\newcommand{\Alg}{\mathcal{A}}
\newcommand{\Idem}{\mathcal{I}}

\newcommand{\CFD}{\mathit{CFD}}
\newcommand{\CFDD}{\mathit{CFDD}}
\newcommand{\CFA}{\mathit{CFA}}

\newcommand{\CFDA}{\mathit{CFDA}}
\newcommand{\CFDAa}{\widehat{\CFDA}}
\newcommand{\CFAA}{\mathit{CFAA}}
\newcommand{\CFAAa}{\widehat{\CFAA}}
\newcommand{\CFDa}{\widehat{\CFD}}

\newcommand{\CFAa}{\widehat{\CFA}}
\newcommand{\tCFDa}{\underline{\CFDa}}
\newcommand{\tCFAa}{\underline{\CFAa}}
\newcommand{\CFDDa}{\widehat{\CFDD}}

\newcommand{\DD}{$\mathit{DD}$}
\newcommand{\DA}{$\mathit{DA}$}

\newcommand{\DT}{\boxtimes}
\newcommand{\BSD}{\widehat{\mathit{BSD}}}
\newcommand{\BSA}{\widehat{\mathit{BSA}}}
\newcommand{\BSDA}{\widehat{\mathit{BSDA}}}
\newcommand{\BSDD}{\widehat{\mathit{BSDD}}}
\newcommand{\BSAA}{\widehat{\mathit{BSAA}}}
\newcommand{\DDid}{\mathit{DDId}}
\newcommand{\SFH}{\mathit{SFH}}
\newcommand{\SFC}{\mathit{SFC}}

\newcommand{\HD}{\mathcal{H}}
\newcommand{\alphas}{\boldsymbol{\alpha}}
\newcommand{\betas}{\boldsymbol{\beta}}

\newcommand{\x}{\mathbf{x}}
\newcommand{\y}{\mathbf{y}}
\newcommand{\z}{\mathbf{z}}
\newcommand{\HB}{H}

\newcommand{\SpinC}{\mathrm{spin}^c}

\newcommand{\PMC}{\mathcal{Z}}
\newcommand{\AZ}{\mathsf{AZ}}

\newcommand{\TW}{\mathcal{TW}} 

\newcommand{\TC}{\mathsf{TC}}

\newcommand{\cM}{\mathcal{M}}

\newcommand{\ModCat}{\mathsf{Mod}}

\newcommand{\nbd}{\mathrm{nbd}}

\theoremstyle{plain}

\numberwithin{equation}{section}
\newtheorem{theorem}{Theorem}[section]

\newtheorem{proposition}[theorem]{Proposition}
\newtheorem{lemma}[theorem]{Lemma}
\newtheorem{corollary}[theorem]{Corollary}

\newtheorem{convention}[theorem]{Convention}
\newtheorem{definition}[theorem]{Definition}
\newtheorem{construction}[theorem]{Construction}

\theoremstyle{definition}

\theoremstyle{remark}

\newtheorem{remark}[theorem]{Remark}

\definecolor{darkgreen}{rgb}{0,.25,0}
\definecolor{darkred}{rgb}{.25,0,0}
\definecolor{pink}{rgb}{1,.08,.575}

\makeatletter
\providecommand\@dotsep{5}
\def\listtodoname{List of Todos}
\def\listoftodos{\@starttoc{tdo}\listtodoname}
\makeatother

\newcommand{\bLambda}{\mathbf{\Lambda}} 
\newcommand{\bFrac}{\mathit{FoF}} 
\newcommand{\btau}{\boldsymbol{\tau}} 

\newcommand{\nobc}{%
  \ooalign{\hidewidth \scriptsize$\mathrm{bc}$\hidewidth\cr\rule[.3ex]{2ex}{1pt}}}
\newcommand{\inobc}{i_{\nobc}}
\newcommand{\ibc}{i_{\mathrm{bc}}}
\newcommand{\cupplus}{{\setbox0\hbox{\large$\cup$}\rlap{\hbox to \wd0{\hss\raisebox{3pt}{\tiny$+$}\hss}}\box0}}
\newcommand{\cupminus}{{\setbox0\hbox{\large$\cup$}\rlap{\hbox to \wd0{\hss\raisebox{3pt}{\tiny$-$}\hss}}\box0}}

\begin{document}
\title{Bordered Floer homology and incompressible surfaces}

\author{Akram Alishahi}
\address{Department of Mathematics, Columbia University, New York, NY 10027}
\email{\href{mailto:alishahi@math.columbia.edu}{alishahi@math.columbia.edu}}

\author{Robert Lipshitz}
\thanks{RL was supported by NSF Grant DMS-1642067.}
\address{Department of Mathematics, University of Oregon, Eugene, OR 97403}
\email{\href{mailto:lipshitz@uoregon.edu}{lipshitz@uoregon.edu}}


\date{\today}

\begin{abstract}
We show that bordered Heegaard Floer homology detects homologically essential
compressing disks and bordered-sutured Floer homology detects partly
boundary parallel tangles and bridges, in natural ways. For example,
there is a bimodule $\bLambda$ so that the tensor product of
$\CFDa(Y)$ and $\bLambda$ is $\Hom$-orthogonal to $\CFDa(Y)$ if and
only if the boundary of $Y$ admits a homologically essential
compressing disk. In the process, we sharpen a nonvanishing result of
Ni's. We also extend Lipshitz-Ozsv\'ath-Thurston's ``factoring''
algorithm for computing $\HFa$~\cite{LOT4} to compute bordered-sutured
Floer homology, to make both results on detecting essential
incompressibility practical. In particular, this makes computing Zarev's tangle invariant manifestly combinatorial.
\end{abstract}

\maketitle

\tableofcontents

\newcommand{\cY}{\mathcal{Y}}
\section{Introduction}\label{sec:intro}
Around the turn of the century, Ozsv\'ath and Szab\'o introduced a new family of 3-manifold invariants, now called \emph{Heegaard Floer homology}~\cite{OS04:HolomorphicDisks}.  These invariants, which are isomorphic to Seiberg-Witten Floer homology~\cite{Taubes10:ECH-SW1,Taubes10:ECH-SW2,Taubes10:ECH-SW3,Taubes10:ECH-SW4,Taubes10:ECH-SW5,KutluhanLeeTaubes:HFHMI,KutluhanLeeTaubes:HFHMII,KutluhanLeeTaubes:HFHMIII,KutluhanLeeTaubes:HFHMIV,KutluhanLeeTaubes:HFHMV,ColinGhigginiHonda11:HF-ECH-1,ColinGhigginiHonda11:HF-ECH-2,ColinGhigginiHonda11:HF-ECH-3}, have led to remarkable new theorems in 3-manifold topology. Instrumental in their success is some of the geometric information they are known to carry: Heegaard Floer homology detects the Thurston norm on the homology of 3-manifolds~\cite{OS04:ThurstonNorm} and whether 3-manifolds fiber over the circle with fiber in a given homology class~\cite{Ni09:FiberedMfld}. The goal of this paper is to develop some more, albeit related, geometric properties detected by Heegaard Floer homology.

While studying the Cosmetic Surgery problem, Ni proved a non-vanishing theorem for Heegaard Floer homology twisted by a closed $2$-form $\omega$~\cite{Ni:spheres}; this is quoted as Theorem~\ref{Ni:non-vanish} below. (See also~\cite[Theorem 2.2]{HeddenNi10:small},~\cite[Corollary 5.2]{HeddenNi13:detects}.) Using a little homological algebra, we strengthen Ni's result to prove:
\begin{theorem}\label{thm:detect-S2}Fix a closed, oriented
  $3$-manifold  $Y$ and a closed $2$-form $\omega\in\Omega^2(Y)$.
  Then $\tHFa(Y;\Lambda_\omega)=0$ if and only if $Y$ contains a $2$-sphere $S$ so that $\int_{S}\omega\neq 0$.
\end{theorem}

By contrast, in the untwisted case, results of
Ozsv\'ath-Szab\'o~\cite[Proposition 5.1]{OS04:HolDiskProperties} and
Hedden-Ni~\cite[Theorem 2.2]{HeddenNi10:small} imply that Heegaard
Floer homology is always non-vanishing:

\begin{theorem}\label{thm:HF-nonvanish}
  For any closed, oriented $3$-manifold $Y$ we have $\HFa(Y)\neq 0$.
\end{theorem}
(This result is well-known to experts, but we do not know a
reference. An extension to link complements is Theorem~\ref{thm:SFH-nonvanishing}, below.)

Recall that a \emph{compressing disk} for a closed surface $\Sigma$ in
a $3$-manifold $Y$ is an embedded, closed disk $D$ in $Y$ such that
$\bdy D=D\cap \Sigma$ is an essential curve in $\Sigma$. A compressing
disk $D$ is \emph{homologically essential} if $0\neq [\bdy D]\in
H_1(\Sigma)$. (See also Remark~\ref{rem:weak-compress}.) If no
homologically essential compressing disk exists then we say $\bdy Y$
is \emph{essentially incompressible}.

Let $Y$ be a $3$-manifold with boundary. Using
Theorem~\ref{thm:detect-S2} we show that bordered Floer homology
detects whether $\bdy Y$ is essentially incompressible:
\begin{theorem}\label{thm:detect-incompress}
  Fix any pointed matched circle $\PMC$ representing a surface of genus $k$. There is a type \DA\ bimodule $\lsup{\Alg(\PMC)}\bLambda(\PMC)_{\Alg(\PMC)}$ with the following property.
  Suppose $\bdy Y$ is a surface of genus $k$ and choose any parametrization
  $\phi\co F(-\PMC)\to\bdy Y$. Then
  \begin{equation}\label{eq:detect-incompres}
    H_*\Mor_{\Alg(\PMC)}\bigl(\lsup{\Alg(\PMC)}\CFDa(Y,\phi),\lsup{\Alg(\PMC)}\bLambda(\PMC)_{\Alg(\PMC)}\DT\lsup{\Alg(\PMC)}\CFDa(Y,\phi)\bigr)=0
  \end{equation}
  if and only if $\bdy Y$ has a homologically essential compressing disk.
\end{theorem}
One could abbreviate Equation~\ref{eq:detect-incompres} by saying that $\CFDa(Y)$ and $\bLambda(\PMC)\DT\CFDa(Y)$ are \emph{$\Hom$-orthogonal}. (The terminology is justified by the fact that the Euler characteristic of the morphism complex is a bilinear form on the Grothendieck group of \dg modules or type $D$ structures. In the case of bordered Floer homology, this bilinear form is closely related to the symplectic form on $H_1(F(\PMC))$~\cite{HLW17:chi}.)

Heegaard Floer homology seems not to detect connected sums, and
consequently it seems unlikely that bordered Floer homology detects
boundary sums.  This is the main reason for the restriction to homologically essential compressing disks in Theorem~\ref{thm:detect-incompress}. (See also Remark~\ref{rem:weak-compress}.)

\begin{convention}
  In the rest of this paper, unless otherwise noted, \emph{compressible} (respectively \emph{incompressible}) means essentially compressible (incompressible), i.e.,  possessing (not possessing) a homologically essential compressing disk.
\end{convention}

Theorem~\ref{thm:detect-incompress} gives a simple procedure to detect incompressibility of surfaces $\Sigma$ in closed $3$-manifolds $Y$, and incompressibility of Seifert surfaces $\Sigma$ for nontrivial knots, as long as $\Sigma$ is not an (embedded) connected sum. Specifically, let $Y'=Y\setminus\nbd(\Sigma)$. Then $\bdy Y'$ is incompressible if and only if $\Sigma$ is incompressible, so apply Theorem~\ref{thm:detect-incompress} to $Y'$.

\begin{remark}
  In the case that $b_1(Y)=g(\bdy Y)=1$, the fact that bordered Floer homology detects whether $\bdy Y$ is incompressible is immediate from a result of Gillespie's~\cite[Corollary 2.7]{Gillespie:Lspace}.
\end{remark}

We also prove a relative version of
Theorem~\ref{thm:detect-incompress}, using Zarev's bordered-sutured
Floer homology~\cite{Zarev09:BorSut}. A \emph{tangle} $(Y,T)$ is a
properly embedded $1$-manifold-with-boundary $T$ in compact, oriented
$3$-manifold $Y$ with connected boundary. Two tangles are equivalent if they are isotopic
fixing the boundary. For a tangle $(Y,T)$, an interval component $T_0$
of $T$ is called \emph{boundary parallel}, if it is isotopic, in
$Y\setminus (T\setminus T_0)$, to an arc in $\bdy Y$. The pair $(Y,T)$
is called \emph{partly boundary parallel} if it has an interval
component which is boundary parallel, and $(Y,T)$ is called \emph{boundary parallel} if $T$ has no closed component and all of its components are boundary parallel. In $B^3$, a tangle is boundary parallel if it consists of a collection of bridges.

\begin{convention}
A tangle is \emph{nullhomologous} if every component is trivial in
$H_1(Y,\bdy Y)$. From now on, tangle means nullhomologous tangle. We also require that any tangle under consideration intersect every component of $\bdy Y$.
\end{convention}
Note that $T$ being nullhomologous implies that every component of $T$
has both endpoints on the same component of $\bdy Y$. Since $T$ is
nullhomologous, for each interval component $T_0$ of $T$ there is an
arc $\gamma$ in $\bdy\nbd(T_0)$ with endpoints on $\bdy Y$, intersecting a merdian of $T_0$ once, and an arc
$\eta$ in $\bdy Y$ with $\bdy \eta=\bdy\gamma$ and
$[\gamma-\eta]=0\in H_1(Y\setminus T_0)$. Any such $\gamma$ is a
\emph{longitude} of $T_0$. Longitudes are not unique up to isotopy rel
boundary: one can twist a longitude around a meridian of $T_0$ to
obtain a new longitude.

\begin{construction}\label{con:tang-to-BS}
  Given a tangle $(Y,T)$, we can make $Y(T)=Y\setminus\nbd(T)$ into a
  bordered-sutured manifold (see Section~\ref{sec:borsut-background}) $\mathcal{Y}_\Gamma(T)$ by:
  \begin{itemize}
  \item declaring $\bdy Y\setminus\nbd(\bdy T)$ to be bordered boundary and $\bdy\nbd(T)$ to be sutured boundary,
  \item placing two meridional sutures around each closed component of
    $T$, 
  \item placing two longitudinal sutures running along each interval component
    of $T$, and 
  \item choosing a parametrization of the bordered boundary of $Y(T)$ by some arc diagram $-\PMC$ for a surface $F(-\PMC)$ with $2n$ boundary components.
  \end{itemize}
  See Figure~\ref{fig:TangToBorSut}. Let $\Gamma$ denote the set of sutures. Note that there are many equally valid choices of longitudinal sutures, so $\Gamma$ is not unique. The region $R_+$ (respectively $R_-$) consists of $n$ rectangles, one per arc of $T$, and $n$ annuli, one per closed component of $T$. Each of $S_+$ and $S_-$ consists of $2n$ arcs, one per component of $\bdy T$.

 Further, for any interval component $T_i$ of $T$, let $\Gamma_i$ be
 the set of sutures obtained from $\Gamma$ by replacing the
 longitudinal sutures for each interval component $T_j$  of
 $T\setminus T_i$ with sutures parallel to the boundary, so that $\bdy
 \Gamma_i=\bdy \Gamma$. For each $T_j$, the new sutures decompose the
 corresponding annulus component of $\bdy\nbd(T)$  into two disks
 (bigons) and an annulus. Again, $\Gamma_i$ is not unique: there are
 two choices of sutures for each $T_j$, $j\neq i$, depending on whether we choose to make $R_+$ or $R_-$ consist of bigons, and the same infinitude of options for $T_i$. Let $\mathcal{Y}_{\Gamma_i}(T)$ be the resulting bordered-sutured manifold for some choice of $\Gamma_i$.
  
\end{construction} 
Note that the tangle $T$ induces a pairing of the boundary components of $F(-\PMC)$.

\begin{figure}
  \centering
  \includegraphics{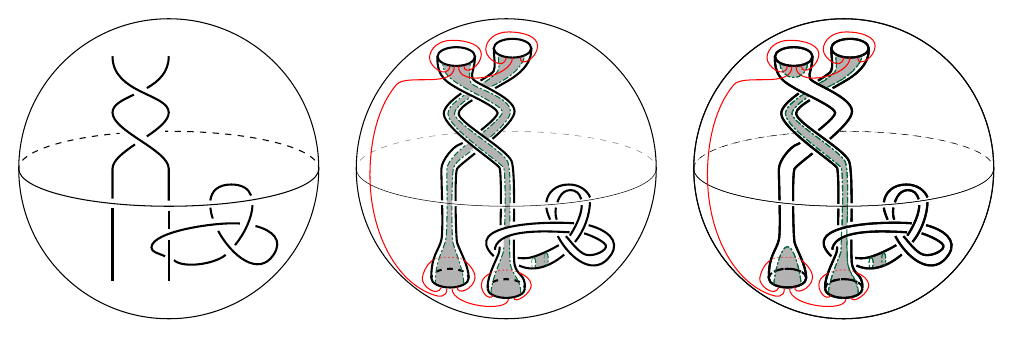}
  \caption{\textbf{From a tangle to a bordered-sutured manifold.} Left: a tangle. Center: the bordered-sutured manifold $\mathcal{Y}_\Gamma(T)$. The sutures $\Gamma$ are \textcolor{darkgreen}{thin, dash-dotted lines}, the region $R_-$ is shaded, and the arcs parameterizing the bordered boundary are \textcolor{red}{thin, solid lines}. The boundary of $\Gamma$ is $\gamma$ and the boundary of $R_-$ is $S_-$ (see Section~\ref{sec:borsut-background}). Right: the manifold $\mathcal{Y}_{\Gamma_1}(T)$.}
  \label{fig:TangToBorSut}
\end{figure}

\begin{theorem}\label{thm:detect-tang}
  For any arc diagram $-\PMC$ as above and pairing of the boundary components of $F(-\PMC)$
  there is a type \DA\ bimodule $\lsup{\Alg(\PMC)}\btau_{\Alg(\PMC)}$ with the following property.
  Given a null-homologous tangle $(Y,T)$, 
  \[
    H_*\Mor_{\Alg(\PMC)}\bigl(\lsup{\Alg(\PMC)}\BSD(\mathcal{Y}_{\Gamma}(T)),\lsup{\Alg(\PMC)}\btau_{\Alg(\PMC)}\DT\lsup{\Alg(\PMC)}\BSD(\mathcal{Y}_{\Gamma}(T))\bigr)=0
  \]
  if and only if $T$ is partly boundary parallel.
  
 Similarly, for any pair of boundary components of $F(-\PMC)$, for instance specified by a choice of interval component $T_i$ of $T$, there is a DA bimodule $\lsup{\Alg(\PMC)}\btau^i_{\Alg(\PMC)}$ so that 
   \begin{equation}\label{eq:bpcompo}
    H_*\Mor_{\Alg(\PMC)}\bigl(\lsup{\Alg(\PMC)}\BSD(\mathcal{Y}_{\Gamma_i}(T)),\lsup{\Alg(\PMC)}\btau^i_{\Alg(\PMC)}\DT\lsup{\Alg(\PMC)}\BSD(\mathcal{Y}_{\Gamma_i}(T))\bigr)=0
  \end{equation}
   if and only if $T_i$ is boundary parallel.
\end{theorem}
(The bimodule $\btau$ consists of a boundary Dehn twist around one end
of each interval component of $T$, while $\btau^i$ is the boundary
Dehn twist around one of the endpoints of $T_i$.)

\begin{corollary}
A tangle $(Y,T)$ with no closed components is boundary parallel if and only if Equation~\ref{eq:bpcompo} holds for all components of $T$.
\end{corollary}

Finally, note that $\CFDa(Y,\phi_i)$ is algorithmically computable
from a Heegaard decomposition of $Y$~\cite{LOT4}. Essentially the same
algorithm also computes bordered-sutured Floer homology, as we explain
in Section~\ref{sec:comp}. The bimodule $\bLambda(\PMC)$ is described
below explicitly, but involves power series, so it is not immediately
obvious if Theorems~\ref{thm:detect-incompress}
and~\ref{thm:detect-tang} give effective criteria. Nonetheless, the
underlying techniques are effective, and Section~\ref{sec:comp} also gives computationally effective versions of Theorems~\ref{thm:detect-incompress} and~\ref{thm:detect-tang}.

This paper is organized as follows. Section~\ref{sec:background} collects some background results on Heegaard Floer homology with Novikov coefficients, sutured Floer homology, bordered Floer homology, and bordered-sutured Floer homology. Section~\ref{sec:non-vanishing} proves the non-vanishing theorems we need for Heegaard Floer homology, Theorems~\ref{thm:detect-S2} and~\ref{thm:HF-nonvanish}, as well as a non-vanishing result for sutured Floer homology, Theorem~\ref{thm:SFH-nonvanishing}. Section~\ref{sec:bordered-detects} shows that bordered Floer homology detects essential compressing disks for 3-manifolds with connected boundary, Theorem~\ref{thm:detect-incompress}. An extension to disconnected boundary is given in Section~\ref{sec:multiple-bdy}. Section~\ref{sec:tangles} proves that bordered-sutured Floer homology detects partly boundary parallel tangles, Theorem~\ref{thm:detect-tang}. The long final section of the paper (Section~\ref{sec:comp}) is devoted to showing that these results are effective and, at least plausibly, computationally useful. The main work in that section is extending the ``factoring'' algorithm for computing $\HFa$~\cite{LOT4} to compute bordered-sutured Floer homology; this extension seems of independent interest.

\emph{Acknowledgments.} We thank Nick Addington, Paolo Ghiggini, Jake Rasmussen, and Sarah Rasmussen for helpful conversations. A similar result to Theorem~\ref{thm:detect-S2} appears in forthcoming work of Juh\'asz-Levine-Marengon. RL also thanks the Isaac Newton Institute for
Mathematical Sciences, Cambridge, for support and hospitality during the
programme ``Homology theories in low-dimensional topology,'' supported by EPSRC
grant no EP/K032208/1, where part of this work undertaken. Finally, we thank the referee for many helpful suggestions.

\section{Background}\label{sec:background}

\subsection{Variants on Novikov coefficients}
Fix a closed $3$-manifold $Y$ and a field $\Field$. Let $\Lambda$
denote the universal Novikov field with coefficients in $\Field$, so
elements of $\Lambda$ are formal sums
$\sum_{a_i\to\infty}c_{a_i}T^{a_i}$ where $\{a_i\}_{i=0}^\infty$ is a
sequence of real numbers with $\lim_{i\to\infty}a_i=\infty$,
$c_{a_i}\in\Field$, and $T$ is a formal variable. 
(The $c_{a_i}$ may be $0$ so, in particular, finite linear combinations of $T$-powers are allowed.)
Any closed $2$-form $\omega$ induces a map $e_\omega\co H_2(Y)\to \RR$ 
by setting $e_\omega(A)=\int_A\omega$. 

Write elements of the group ring $\Field[H_2(Y)]$ as formal finite sums
$\sum_i n_ie^{\alpha_i}$ where $n_i\in\Field$, $\alpha_i\in H_2(Y)$,
and $e$ is just notation.
Using the map $e_\omega$ one may define a map $\psi_{\omega}\co \Field[H_2(Y)]\to\Lambda$ as
\[\psi_{\omega}(e^{\alpha})=T^{e_{\omega}(\alpha)}.\]
The map $\psi_\omega$ gives the Novikov field $\Lambda$ the structure of a $\Field[H_2(Y)]$-module, denoted by $\Lambda_\omega$. Thus, we can consider Heegaard Floer homology with coefficients in $\Lambda_\omega$~\cite[Section 8]{OS04:HolDiskProperties} and, in particular, $\tHFa(Y;\Lambda_\omega)$. Up to isomorphism, $\tHFa(Y;\Lambda_\omega)$ depends only on the cohomology class of $\omega$~\cite[Proposition 2.1]{Ni09:FiberedMfld}.

Let $I_{\omega}=\ker(e_{\omega})$.
The map $\psi_\omega$ induces an injective map 
$\Field[H_2(Y)/I_{\omega}]\to \Lambda_\omega$ and thus a field homomorphism
\[\imath_{\omega}\co \Field(H_2(Y)/I_{\omega})\to\Lambda_\omega\]
where $\Field(H_2(Y)/I_{\omega})$ is the field of fractions of $\Field[H_2(Y)/I_{\omega}]$. Let $R_{\omega}=\Field[H_2(Y)/I_{\omega}]$ and $K_\omega=\Field(H_2(Y)/I_{\omega})$. Since
\[\tHFa(Y;\Lambda_\omega)\cong H_*\left(\tCFa(Y;R_{\omega})\otimes_{R_{\omega}}\Lambda_{\omega}\right)\cong H_*\left(\tCFa(Y;K_{\omega})\otimes_{K_{\omega}}\Lambda_{\omega}\right)
\]
and $K_\omega$ is a field, 
the universal coefficients theorem implies that 
\[
\tHFa(Y;\Lambda_{\omega})\cong \tHFa(Y;K_{\omega})\otimes_{K_{\omega}}\Lambda_{\omega}.
\]
Furthermore, since $K_\omega$ and $\Lambda_{\omega}$ are fields, 
\begin{equation}\label{eq:tdim}
  \dim_{\Lambda_{\omega}}\left(\tHFa(Y;\Lambda_{\omega})\right)=\dim_{K_\omega}\left(\tHFa(Y;K_{\omega})\right).
\end{equation}

\begin{lemma}
  For any closed $2$-form $\omega\in\Omega^{2}(Y)$, we have
  \[
    \dim_{\Lambda_{\omega}}\left(\tHFa(Y;\Lambda_{\omega})\right)\ge\dim _{\Field(H_2(Y))}\left(\tHFa(Y;\Field(H_2(Y)))\right).
  \]
\end{lemma}

\begin{proof}[Proof, thanks to N.~Addington]
  The chain group $\CFa(Y;\Field[H_2(Y)])$ is the
free $\Field[H_2(Y)]$-module generated by the intersection points of
$\mathbb{T}_{\alpha}$ and $\mathbb{T}_{\beta}$. Let
$N=\#(\mathbb{T}_{\alpha}\cap\mathbb{T}_{\beta})$. Then the
differential $\bdy$ is represented by an $N\times N$
matrix $D\in M_{N\times N}(\Field[H_2(Y)])$.

For any subgroup $I$ of $H_2(Y)$ let $\rank_{\Field[H_2(Y)/I]}(D)$ be the largest $n$ such that there is an $n\times n$ minor of $D$ which has a nonvanishing determinant in $\Field[H_2(Y)/I]$. Thus, 
\[\rank_{\Field[H_2(Y)]}(D)\ge\rank_{\Field[H_2(Y)/I]}(D).\]
On the other hand, if the quotient group $H_2(Y)/I$ has no torsion then
\[\dim_{\Field(H_2(Y)/I)}\left(\tHFa(Y;\Field(H_2(Y)/I))\right)=N-2\rank_{\Field[H_2(Y)/I]}(D).\]
Therefore, for such $I$, $$\dim_{\Field(H_2(Y)/I)}\left(\tHFa(Y;\Field(H_2(Y)/I))\right)\ge\dim_{\Field(H_2(Y))}\left(\tHFa(Y;\Field(H_2(Y)))\right)$$
and the claim follows from the Equation~\eqref{eq:tdim}.
\end{proof}

\begin{definition}
A closed $2$-form $\omega$ is \emph{generic} if the corresponding map $e_\omega$ is injective. 
\end{definition}

Note that the set of generic $2$-forms is dense (in fact, co-meager).

\begin{corollary}\label{HFdim-Novikov}
Let $\omega\in\Omega^2(Y)$ be a generic closed $2$-form. Then for any closed $2$-form $\omega'\in\Omega^2(Y)$ we have 
$$\dim _{\Lambda_{\omega'}}(\tHFa(Y;\Lambda_{\omega'}))\ge\dim _{{\Lambda_\omega}}(\tHFa(Y;\Lambda_{\omega})).$$
\end{corollary}

We conclude this section by recalling the K\"unneth theorem for Novikov coefficients (stated in~\cite[Section 2.3]{Ni09:FiberedMfld}):
\begin{lemma}\label{lem:Novikov-Kunneth}
  Suppose that $Y\cong Y_1\# Y_2$, $\omega_i\in H^2(Y_i;\RR)$, and $\omega=\omega_1\#\omega_2$. Then
  \[
    \tHFa(Y;\Lambda_\omega)\cong\tHFa(Y_1;\Lambda_{\omega_1})\otimes_\Lambda\tHFa(Y_2;\Lambda_{\omega_2})
  \]
  as $\Lambda$-vector spaces.
\end{lemma}
\begin{proof}
  This follows from the definitions of the twisted complexes and the (easy) proof of the K\"unneth theorem for untwisted $\HFa$~\cite[Proposition 6.1]{OS04:HolDiskProperties}.
\end{proof}

\subsection{Sutured Floer homology} For the reader's convenience, we recall a few of the basic definitions and properties of sutured Floer homology.

A \emph{sutured manifold} (without toroidal sutures) is an
oriented $3$-manifold-with-boundary $Y$ along with a set $\Gamma$ of
pairwise disjoint circles in $\bdy Y$, called \emph{sutures}, which
divide $\bdy Y$ into positive and negative regions, denoted
$R_+(\Gamma)$ and $R_-(\Gamma)$, so that $\bdy R_+(\Gamma)=\bdy
R_-(\Gamma)=\Gamma$. Note that we may oriented the components of
$\Gamma$ by the induced orientation from $R_+(\Gamma)$. The sutured
manifold $(Y,\Gamma)$ is called \emph{balanced} if
$\chi(R_+(\Gamma))=\chi(R_-(\Gamma))$ and all components of $\bdy Y$
intersect $\Gamma$, and $(Y,\Gamma)$ is called \emph{taut} if
$Y$ is irreducible and $R_+(\Gamma)$ and $R_-(\Gamma)$ are incompressible and Thurston-norm minimizing in $H_2(Y,\Gamma)$. Note that if $\Gamma$ intersects all components of $\bdy Y$, $Y$ is irreducible, and $R_+(\Gamma)$ and $R_-(\Gamma)$ are Thurston-norm minimizing in $H_2(Y,\Gamma)$, then $(Y,\Gamma)$ is taut if and only if every annular component of $R_+(\Gamma)$ and $R_-(\Gamma)$ is incompressible.  

Given a sutured manifold $(Y,\Gamma)$, the sutured Floer complex
$\SFC(Y,\Gamma)$ is defined as an extension of
$\CFa$~\cite{Juhasz06:Sutured}. Unlike $\HFa(Y)$, sutured Floer
homology can vanish:

\begin{theorem}~\cite[Proposition 9.18]{Juhasz06:Sutured}\label{thm:vanishSFH}
  If an irreducible balanced sutured manifold $(Y,\Gamma)$ is not taut then $\SFH(Y,\Gamma)=0$.
\end{theorem}

If one decomposes a balanced sutured manifold along a surface then the sutured Floer homology of the resulting sutured manifold is a direct summand of the sutured Floer homology of the original sutured manifold~\cite[Theorem 1.3]{Juhasz08:SuturedDecomp}. As a corollary of this, Juh\'asz shows: 
\begin{theorem}\cite[Theorem 1.4]{Juhasz08:SuturedDecomp}\label{thm:nonvanishSFH} If 
$(Y,\Gamma)$ is a taut, balanced sutured manifold then $\SFH(Y,\Gamma)\neq 0$.
\end{theorem}

\subsection{Bordered Floer homology}
In this section, we review some aspects of bordered Floer homology, mainly to
fix notation.

\subsubsection{Bordered basics}
In bordered Floer homology, a surface is represented by a
\emph{pointed matched circle} $\PMC=(Z,\mathbf{a},M,z)$ where $Z$ is
an oriented circle, $\mathbf{a}$ is $4k$ points in $Z$, $M$ is a
pairing of points in $\mathbf{a}$, and $z\in Z\setminus\mathbf{a}$. There is a requirement that attaching $1$-handles to $[0,1]\times Z$ along $\{0\}\times\mathbf{a}$, according to the matching $M$, gives a surface of genus $k$ (with two boundary components); the result of filling in these two boundary components with disks is denoted $F(\PMC)$~\cite[Section 3.2]{LOT1}. Given a pointed matched circle $\PMC$, reversing the orientation of $Z$ gives a new pointed matched circle $-\PMC$, and $F(-\PMC)=-F(\PMC)$, the orientation reverse of $F(\PMC)$.

Associated to a pointed matched circle $\PMC$ is a \dg algebra $\Alg(\PMC)$ over $\Ftwo$. A basis for $\Alg(\PMC)$ over $\Ftwo$ is given by \emph{strand diagrams}, which consist of a subset of $\mathbf{a}$ and a collection of intervals in $Z\setminus \{z\}$ with endpoints on $\mathbf{a}$, satisfying some conditions. In particular, a strand diagram $b$ has a \emph{support} $[b]\in H_1(Z\setminus \{z\},\mathbf{a})\cong \ZZ^{4k-1}$. Also, corresponding to any subset $\mathbf{s}$ of the matched pairs in $\mathbf{a}$ is a \emph{basic idempotent} $I(\mathbf{s})$. These basic idempotents span a subalgebra $\Idem(\PMC)\subset\Alg(\PMC)$; in fact, as a ring, $\Idem(\PMC)\cong \Ftwo^{2^{2k}}$.

A \emph{bordered $3$-manifold} is a $3$-manifold $Y$ together with a homeomorphism $\phi\co F(\PMC)\to \bdy Y$ for some pointed matched circle $\PMC$. Given a bordered $3$-manifold $Y$, there is a corresponding left \dg module $\lsup{\Alg(-\PMC)}\CFDa(Y)$ over $\Alg(-\PMC)$~\cite[Section 6.1]{LOT1} and right $\Ainf$-module $\CFAa(Y)_{\Alg(\PMC)}$ over $\Alg(\PMC)$~\cite[Section 7.1]{LOT1}. In fact, $\CFDa(Y)$ has a special form: it is a \emph{type $D$ structure} or \emph{twisted complex}, a generalization of projective modules.

The bordered invariants relate to the Heegaard Floer complex $\CFa$ via the \emph{pairing theorem}: given bordered $3$-manifolds $(Y_0,\phi_0\co F(\PMC)\to\bdy Y_0)$ and $(Y_1,\phi_1\co F(-\PMC)\to\bdy Y_1)$ there is a closed $3$-manifold $Y_0\thinspace \lsub{\phi_0}\!\cup_{\phi_1}\!Y_1$ obtained by gluing the boundaries together as prescribed by $\phi_0$ and $\phi_1$, and a chain homotopy equivalence
\[
  \CFa(Y_0\thinspace \lsub{\phi_0}\!\!\cup_{\phi_1}\!Y_1)\simeq \CFAa(Y_0,\phi_0)_{\Alg(\PMC)}\DT\lsup{\Alg(\PMC)}\CFDa(Y_1,\phi_1)
\]
\cite[Theorem 1.2]{LOT1}. Here $\DT$ is a version of the (derived) tensor product between an $\Ainf$-module and a type $D$ structure~\cite[Section 2.4]{LOT1}. Subscripts indicate $\Ainf$-actions, and superscripts type $D$ structures, so $\DT$ satisfies a kind of Einstein convention.

To keep notation simple, we will generally denote $\CFAa(Y_0,\phi_0)$ by
$\CFAa(Y_0)$, $\CFDa(Y_1,\phi_1)$ by $\CFDa(Y_1)$, and
$Y_0\thinspace \lsub{\phi_0}\!\cup_{\phi_1}\!Y_1$ by $Y_0\cup_\bdy Y_1$.

There are also dualities relating $\CFDa(Y)^*$, the dual type $D$
structure to $\CFDa(Y)$, to $\CFAa(-Y)$, which lead to
pairing theorems in terms of morphism complexes: with notation as above,
\begin{align*}
  \CFa(Y_0\cup_\bdy Y_1)&\simeq \Mor^{\Alg(\PMC)}(\CFDa(-Y_0),\CFDa(Y_1))\\
  &\simeq\Mor_{\Alg(-\PMC)}(\CFAa(-Y_0),\CFAa(Y_1))
\end{align*}
\cite[Theorem 1]{LOTHomPair}.

There are generalizations to manifolds with multiple boundary
components~\cite{LOT2}. In particular, given a $3$-manifold $Y$ with
two boundary components, $\bdy_0Y$ and $\bdy_1Y$, a framed arc from
$\bdy_0Y$ to $\bdy_1Y$, and homeomorphisms $F(\PMC_0)\to \bdy_0Y$ and
$F(\PMC_1)\to\bdy_1Y$ there are corresponding bimodules
$\CFAAa(Y)_{\Alg(\PMC_0),\Alg(\PMC_1)}$,
$\lsup{\Alg(-\PMC_0)}\CFDAa(Y)_{\Alg(\PMC_1)}$, and
$\lsup{\Alg(-\PMC_0),\Alg(-\PMC_1)}\CFDDa(Y)$. Here, for instance,
$\CFAAa(Y)$ is an $\Ainf$-bimodule with two right actions, by
$\Alg(\PMC_0)$ and $\Alg(\PMC_1)$. (The appearance of bimodules with
commuting right or commuting left actions is an artifact of the
original definitions, and is sometimes convenient, sometimes
inconvenient.) One can drill out a neighborhood of the framed arc to
obtain a bordered $3$-manifold $Y_{\mathrm{dr}}$ with one boundary
component. The bimodules $\CFDDa(Y)$ and $\CFAAa(Y)$ are determined by
$\CFDa(Y_{\mathrm{dr}})$ and $\CFAa(Y_{\mathrm{dr}})$ by induction and
restriction functors associated to maps of the algebras~\cite[Section 6]{LOT2}. The bimodule $\CFDAa(Y)$ is also determined by, say, $\CFAa(Y_{\mathrm{dr}})$, but in a more complicated way. These satisfy analogues of the pairing theorem, where one always glues type $D$ boundaries to type $A$ boundaries~\cite[Theorems 11 and 12]{LOT2}.

\subsubsection{Bordered-sutured Floer homology}\label{sec:borsut-background}

Next, we recall some definitions and properties of bordered-sutured
Floer homology~\cite{Zarev09:BorSut}, as well as the reformulation of
the bordered-sutured pairing theorem in terms of morphism complexes~\cite{Zarev:JoinGlue}.

A \emph{sutured surface} $(F,\gamma)$ is a compact, oriented
surface-with-boundary $F$, with no closed components, together with a finite set $\gamma$ of marked points on $\bdy F$, such that $\gamma$ intersects all boundary components of $F$ and splits $\bdy F$ as $\bdy F=S_+(\gamma)\cup_{\gamma} S_-(\gamma)$ with $\bdy S_+(\gamma)=\bdy S_-(\gamma)=\gamma$. Sutured surfaces can be parametrized by arc diagrams,
a generalization of pointed matched circles. An \emph{arc diagram}
$\PMC=(Z,\mathbf{a},M)$ is a finite collection $Z$ of oriented, closed
intervals together with a set $\mathbf{a}$ of $2k$ points in $Z$ and a
pairing $M$ of $\mathbf{a}$. It is required that the pairing
$M$ is \emph{nondegenerate}, i.e., attaching $1$-handles to
$[0,1]\times Z$ along $\{0\}\times \mathbf{a}$ as specified by $M$
results in a surface-with-boundary $F$ such that $\{1\}\times \bdy Z$
intersects all boundary components of $F$. This surface along with the
set of points $\gamma=\{1\}\times \bdy Z$ and the division
$S_+(\gamma)=\{1\}\times Z$ and $S_-(\gamma)$ its complement, is a
sutured surface denoted by $F(\PMC)$.

Let $\bar{F}(\PMC)$ denote the sutured surface obtained from $F(\PMC)$
by reversing the orientation of the underlying surface and switching
$S_+$ and $S_-$. Let $-F(\PMC)$ be the sutured surface obtained from
$F(\PMC)$ by reversing the orientation of the underlying surface but
not switching $S_+$ and $S_-$. So, if we think of $\gamma$ as oriented as the boundary of $S_+$ then the orientation of $\gamma$ in $-F(\PMC)$ is reversed.

Given sutured surfaces $(F,\gamma)$ and $(F',\gamma')$, a \emph{sutured cobordism} $(Y,\Gamma)$ from $(F,\gamma)$ to $(F',\gamma')$ is an oriented $3$-manifold $Y$ with boundary, together with a decomposition $\bdy Y\cong \bdy_s Y \cup (-F\amalg F')$ and a properly embedded $1$-manifold-with-boundary $\Gamma\subset\bdy_s Y$ 
  satisfying the following:
\begin{itemize}
\item $Y$ has no closed components,
\item $\pi_0(\Gamma)\to\pi_0(\bdy_sY)$ is surjective and $\bdy \Gamma=\gamma\amalg \gamma'$, and
\item $\Gamma$ divides  $\bdy_sY=R_+\cup_\Gamma R_-$ such that $\bdy R_+\cap\bdy R_-=\Gamma$. Furthermore, $R_\bullet\cap F=S_\bullet(\gamma)$ and $R_\bullet\cap F'=S_\bullet(\gamma')$ for $\bullet\in\{+,-\}$.
\end{itemize}
A sutured cobordism from the empty set to the empty set is a sutured manifold.
Given  a sutured cobordism $(Y,\Gamma)$, let $-(Y,\Gamma)$ be the
cobordism obtained by reversing the orientation of $Y$ and let
$m(Y,\Gamma)$, the \emph{mirror} of $(Y,\Gamma)$, be the cobordism obtained by reversing the orientation of $Y$ and switching the roles of $R_+$ and $R_-$.

Given arc diagrams $\PMC$ and $\PMC'$, a sutured cobordism $(Y,\Gamma)$ from $F(\PMC)$ to $F(\PMC')$ is called a \emph{bordered-sutured cobordism}. 
The subspace $-F(\PMC)\amalg F(\PMC')$ of $\bdy Y$ is called the \emph{bordered boundary} and $\bdy_s Y$ is called the \emph{sutured boundary}. Note that the bordered boundary of  $-(Y,\Gamma)$ is $-F(\PMC')\amalg F(\PMC)$ while the bordered boundary of $m(Y,\Gamma)$ is  $-\bar{F}(\PMC)\amalg\bar{F}(\PMC')$.

If one of $\PMC$ or $\PMC'$ is empty then 
$(Y,\Gamma)$ is called a \emph{bordered-sutured manifold}.

\begin{definition}
A \emph{bordered-sutured Heegaard diagram} is a $4$-tuple
$(\Sigma,\alphas,\betas,\mathbf{z},\bdy_L\Sigma)$ where:
\begin{itemize}
\item $\Sigma$ is a compact surface with boundary.
\item $\mathbf{z}=\mathbf{z}^c\cup\mathbf{z}^a$ is a subspace of $\bdy
  \Sigma$ such that $\mathbf{z}^c$ is a disjoint union of some
  components of $\bdy \Sigma$ and $\mathbf{z}^a$ is a finite set of
  disjoint arcs in $\bdy \Sigma\setminus\mathbf{z}^c$. 
\item $\bdy_L\Sigma$ is a union of some of the connected components of $\bdy\Sigma\setminus\mathrm{int}(\mathbf{z})$. We let $\bdy_R\Sigma$ be the union of the remaining components of $\bdy\Sigma\setminus\mathrm{int}(\mathbf{z})$.
\item $\betas$ is a union of pairwise disjoint circles in the interior of $\Sigma$.
\item $\alphas=\alphas^{a,L}\amalg \alphas^{a,R}\amalg \alphas^{c}$ is a union of pairwise disjoint, properly embedded arcs and circles in $\Sigma$, so that for $\bullet\in\{L,R\}$, each element of $\alphas^{a,\bullet}$ is an arc with boundary on $\bdy_\bullet\Sigma$, while each element of $\alphas^c$ is a circle.
\end{itemize}
We require that  $\mathbf{z}$ intersects every component of
$\bdy\Sigma$, $\Sigma\setminus\betas$, and $\Sigma\setminus\alphas$.
\end{definition}
Any bordered-sutured Heegaard diagram $\HD$ specifies two arc diagrams:
\[
\begin{split}
&\PMC_L(\HD)=(\bdy_L\Sigma,\alphas^{a,L}\cap\bdy_L\Sigma,M_L)\\
&\PMC_R(\HD)=(\bdy_R\Sigma,\alphas^{a,R}\cap\bdy_R\Sigma,M_R)
\end{split}
\] 
where $M_L$ and $M_R$ are the matchings that pair the endpoints of each arc in $\alphas^{a,L}$ and $\alphas^{a,R}$, respectively. 

Given a bordered-sutured Heegaard diagram $\HD$ we get a
bordered-sutured cobordism $\mathcal{Y}(\HD)$ from
$F(-\PMC_L(\HD))$ to $F(\PMC_R(\HD))$ as follows. Let $Y$ be the $3$-manifold obtained from $\Sigma\times [0,1]$ by attaching $2$-handles along the $\alpha$-circles in $\Sigma\times\{0\}$ and $\beta$-circles in $\Sigma\times\{1\}$. The $\alpha$-arcs specify two embedded sutured surfaces 
\[\begin{split}
&F(\PMC_L(\HD))=\left(\bdy_L\Sigma\times[0,1]\cup\nbd(\alphas^{a,L})\times\{0\},(\bdy_L\Sigma\cap\mathbf{z})\times \{1\}\right)\\
&F(\PMC_R(\HD))=\left(\bdy_R\Sigma\times[0,1]\cup\nbd(\alphas^{a,R})\times\{0\},(\bdy_R\Sigma\cap\mathbf{z})\times \{1\}\right)
\end{split}
\] 
in $\bdy Y$. Finally, let $\Gamma=\mathbf{z}\times\{1\}$, and let $R_+$ (respectively $R_-$) be the closure of the union of the components of $\bdy Y\setminus (F(\PMC_L)\cup F(\PMC_R)\cup \Gamma)$ that intersect $\Sigma\times\{1\}$ (respectively $\Sigma\times\{0\}$). Thus, $R_+$ is the result of performing surgery on the $\beta$-circles in $\HD$ and $R_-$ is the result of performing surgery on the $\alpha$-circles and deleting neighborhoods of the $\alpha$-arcs.

Associated to an arc diagram $\PMC$ is a \dg algebra $\Alg(\PMC)$,
spanned by \emph{strand diagrams}; if $\PMC$ corresponds to a pointed
matched circle, $\Alg(\PMC)$ is the same as the bordered algebra
associated to the pointed matched circle.
Next, let $\mathcal{Y}$ be a bordered-sutured cobordism from $F(\PMC)$
to $F(\PMC')$. To $\mathcal{Y}$, Zarev associates {\it bordered-sutured bimodules} $\lsup{\Alg(\PMC),\Alg(-\PMC')}\BSDD(\mathcal{Y})$, $\lsup{\Alg(\PMC)}\BSDA(\mathcal{Y})_{\Alg(\PMC')}$, and $\BSAA(\mathcal{Y})_{\Alg(-\PMC),\Alg(\PMC')}$~\cite{Zarev09:BorSut}. If $\PMC$ is empty then, $\BSDA(\mathcal{Y})$ is a right $\Alg_\infty$-module over $\Alg(\PMC')$ denoted $\BSA(\mathcal{Y})_{\Alg(\PMC')}$, while if $\PMC'$ is empty, $\BSDA(\mathcal{Y})$ is a left type $D$ structure over $\Alg(\PMC)$ denoted $\lsup{\Alg(\PMC)}\BSD(\mathcal{Y})$. Furthermore, if both $\PMC$ and $\PMC'$ are empty, i.e., $\mathcal{Y}$ is a sutured manifold, then $\BSDA(\mathcal{Y})$ is the sutured Floer complex. Similar statements are true for the other bimodules. Finally, if $\mathcal{Y}$ is a bordered-sutured manifold with bordered boundary $\bar{F}(\PMC)$, then $\lsup{\Alg(\PMC)}\BSD(\mathcal{Y})$ is a type $D$ structure over $\Alg(\PMC)$ and $\BSA(\mathcal{Y})_{\Alg(-\PMC)}$ is an $A_{\infty}$-module over $\Alg(-\PMC)$. A similar statement is true for bimodules associated with bordered-sutured cobordisms from $\bar{F}(\PMC)$ or to $\bar{F}(\PMC')$.

Suppose that $\mathcal{Y}$ is a cobordism from $F(\PMC_1)$ to
$F(\PMC_2)$. We can also regard $\mathcal{Y}$ as a cobordism from
$F\bigl(\PMC_1\amalg (-\PMC_2)\bigr)$ to the empty set. It is evident
from the definitions (which we have not given) that there is an isomorphism $\Alg(\PMC_1\amalg (-\PMC_2))\cong \Alg(\PMC_1)\otimes_{\Ftwo}\Alg(-\PMC_2)$. Furthermore, a type \DD\ structure over $\Alg(\PMC_1)$ and $\Alg(-\PMC_2)$ is exactly the same as a type $D$ structure over $\Alg(\PMC_1)\otimes_{\Ftwo} \Alg(-\PMC_2)$. With respect to these identifications,
\[
  \lsup{\Alg(\PMC_1),\Alg(-\PMC_2)}\BSDD(\mathcal{Y})\cong \lsup{\Alg(\PMC_1)\otimes_{\Ftwo}\Alg(-\PMC_2)}\BSD(\mathcal{Y}).
\]
So, in the bordered-sutured setting, one is justified in talking about the module $\BSD$ associated to a cobordism from $F(\PMC_1)$ to $F(\PMC_2)$. Also, computing the modules $\BSD$ associated to all bordered-sutured manifolds is equivalent to computing the bimodules $\BSDD$ associated to all bordered-sutured cobordisms. (Similar statements hold for $\BSA$ and $\BSAA$, but are slightly complicated by the fact that an $\Ainf$ $(\Alg_1,\Alg_2)$-bimodule is not quite the same as an $\Ainf$-module over $\Alg_1\otimes_{\Ftwo}\Alg_2$.)

Given a bordered-sutured Heegaard diagram $\HD$,  we may switch $\alpha$ circles and arcs with $\beta$ circles to get a diagram with $\beta$ arcs, called a \emph{$\beta$-bordered-sutured Heegaard diagram}, and denoted by $\HD^{\beta}$. This diagram represents the mirror $m(\mathcal{Y}(\HD))$. 

The algebras $\Alg(\PMC)$ satisfy $\Alg(-\PMC)=\Alg(\PMC)^\op$ so, for
instance, we can regard $\lsup{\Alg(\PMC)}\BSD(\mathcal{Y})$ as a
right type $D$ structure $\BSD(\mathcal{Y})^{\Alg(-\PMC)}$ over
$\Alg(-\PMC)$. Similarly, $\BSD(m(\mathcal{Y}))^{\Alg(\PMC)}$ is a
left type $D$ structure over $\Alg(-\PMC)$ or a right type $D$
structure over $\Alg(\PMC)$. In fact, since given a Heegaard diagram
$\HD$ for $\mathcal{Y}$, $\HD^\beta$ represents $m(\mathcal{Y})$, it is
immediate from the definitions that 
\begin{equation}\label{eq:BSD-dual}
  \BSD(m(\mathcal{Y}))^{\Alg(\PMC)}\cong\left(\lsup{\Alg(\PMC)}\BSD(\mathcal{Y})\right)^*
\end{equation}
\cite[Proposition 3.11]{Zarev:JoinGlue}. 

Zarev's construction has a pairing theorem as follows:

\begin{theorem}\cite[Theorem 10.4]{Zarev09:BorSut}\label{thm:bordred-sut-pairing}
Let $\PMC_i$ be an arc diagram for $i=1,2,3$. Given bordered-sutured cobordisms $\mathcal{Y}_1$ from $F(\PMC_1)$ to $F(\PMC_2)$ and $\mathcal{Y}_2$ from $F(\PMC_2)$ to $F(\PMC_3)$ we have
 \[
  \lsup{\Alg(\PMC_1)}\BSDA(\mathcal{Y}_1)_{\Alg(\PMC_2)}\DT\lsup{\Alg(\PMC_2)}\BSDA(\mathcal{Y}_2)_{\Alg(\PMC_3)}
  \simeq \lsup{\Alg(\PMC_1)}\BSDA\bigl(\mathcal{Y}_1\cup_{F(\PMC_2)} \mathcal{Y}_2\bigr)_{\Alg(\PMC_3)}.
\]
Similar statements hold for tensor products of the other types of bimodules, as long as one tensors together one type $A$ and one type $D$ action.
In particular, if $\PMC_1$ and $\PMC_3$ are empty, then $\mathcal{Y}_1\cup_{F(\PMC_2)} \mathcal{Y}_2$ is a sutured manifold and 
\[
    \SFC\bigl(\mathcal{Y}_1\cup_{F(\PMC_2)} \mathcal{Y}_2\bigr)\simeq\BSA(\mathcal{Y}_1)_{\Alg(\PMC_2)}\DT\lsup{\Alg(\PMC_2)}\BSD(\mathcal{Y}_2)
\]
\end{theorem}

As in the bordered case, one can reformulate the pairing theorem in
terms of morphism complexes~\cite{Zarev:JoinGlue}. Suppose $\mathcal{Y}_1$ and $\mathcal{Y}_2$ are bordered-sutured manifolds with bordered boundary $F(-\PMC)$. From $m(\mathcal{Y}_1)$ one gets a bordered-sutured manifold with bordered boundary $F(\PMC)$ by gluing the \emph{negative twisting slice} $\TW_{F(-\PMC),-}$. This negative twisting slice is a bordered-sutured cobordism from $\bar{F}(-\PMC)$ to $F(\PMC)$ obtained from $[0,1]\times F(\PMC)$ by making the sutures veer left, the minimum amount possible~\cite[Definition 2.8]{Zarev:JoinGlue}.

There is a particular $(\alpha,\beta)$-bordered Heegaard diagram
for the negative twisting slice $\TW_{F(-\PMC),-}$ with the property that
\begin{equation}\label{eq:BSAA-of-TW}
  \lsub{\Alg(\PMC)}\BSAA(\TW_{F(-\PMC),-})_{\Alg(\PMC)}\simeq \lsub{\Alg(\PMC)}\Alg(\PMC)_{\Alg(\PMC)}
\end{equation}
\cite[Proposition 3.12]{Zarev:JoinGlue}. Note that we always treat the $\beta$-boundary of $\TW_{F(-\PMC),-}$ as corresponding to the left action of $\BSAA(\TW_{F(-\PMC),-})$, while Zarev treats the $\beta$-boundary as the right action. This means the algebras that show up Formula~\eqref{eq:BSAA-of-TW} are the opposites of the algebras in Zarev's result.
The diagram $\TW_{F(-\PMC),-}$ was also discovered by
Auroux~\cite{Auroux10:Bordered} and, in other papers, has been denoted
$\AZ(-\PMC)$~\cite{LOTHomPair}.
Formula~(\ref{eq:BSD-dual}) and Theorem~\ref{thm:bordred-sut-pairing}
imply that
\[
\BSA(m(\mathcal{Y}_1)\cup\TW_{F(-\PMC),-})_{\Alg(\PMC)}\simeq\left(\BSD(\mathcal{Y}_1)^*\right)^{\Alg(\PMC)}\DT\lsub{\Alg(\PMC)}\Alg(\PMC)_{\Alg(\PMC)}.
\]

Recall that for any type $D$
structures $\lsup{\Alg(\PMC)}P$ and $\lsup{\Alg(\PMC)}Q$,
\[
  \Mor^{\Alg(\PMC)}(\lsup{\Alg(\PMC)}P,\lsup{\Alg(\PMC)}Q)=(P^*)^{\Alg(\PMC)}\DT\lsub{\Alg(\PMC)}\Alg(\PMC)_{\Alg(\PMC)}\DT \lsup{\Alg(\PMC)}Q
\]
\cite[Proposition 2.7]{LOTHomPair}.
So, applying the pairing theorem again, we obtain:
\begin{theorem}\label{thm:BS-hom-pair}\cite{Zarev:JoinGlue}
  Let $\mathcal{Y}_1=(Y_1,\Gamma_1)$ and $\mathcal{Y}_2=(Y_2,\Gamma_2)$ be bordered-sutured $3$-manifolds with bordered
  boundaries parametrized by $F(-\PMC)$. Then
  \[
    \Mor^{\Alg(\PMC)}(\lsup{\Alg(\PMC)}\BSD(\mathcal{Y}_1),\lsup{\Alg(\PMC)}\BSD(\mathcal{Y}_2))
    \simeq \SFC(m(\mathcal{Y}_1)\cup\TW_{F(-\PMC),-}\cup \mathcal{Y}_2).
  \]
\end{theorem}

Given an arc diagram $\PMC$ and a component $C$ of $\bdy F(\PMC)$
there is a corresponding bordered-sutured cobordism $\tau_C$ from
$F(\PMC)$ to $F(\PMC)$, the \emph{Dehn twist along $C$}, obtained from
the identity bordered-sutured cobordism $[0,1]\times F(\PMC)$ by performing a right-veering Dehn twist on the sutures in $[0,1]\times C$.

Next, fix an arc diagram $\PMC$ so that $\bdy F(\PMC)$ has $2n$ components $C_1,\dots,C_{2n}$, which we think of as identified in pairs via $C_i\leftrightarrow C_{i+n}$. Let $\lsup{\Alg(\PMC)}\btau_{\Alg(\PMC)}$ be the type \DA\ bimodule corresponding to performing Dehn twists on exactly half of these components, by choosing one from each pair. That is, assuming that for each $i$ we have chosen $C_i$ from $\{C_i,C_{i+n}\}$ then
\[
\begin{split}
  \lsup{\Alg(\PMC)}\btau_{\Alg(\PMC)}&=\lsup{\Alg(\PMC)}\BSDA(\tau_{C_1}\circ \tau_{C_2}\circ\cdots\circ \tau_{C_n})_{\Alg(\PMC)}\\
  &\simeq \lsup{\Alg(\PMC)}\BSDA(\tau_{C_1})_{\Alg(\PMC)}\DT\cdots\DT \lsup{\Alg(\PMC)}\BSDA(\tau_{C_n})_{\Alg(\PMC)}.
  \end{split}
\]
Then, Theorem~\ref{thm:BS-hom-pair} has the following corollary:
\begin{corollary}\label{cor:doubling}
  Let $(Y,T)$ be a tangle and $\mathcal{Y}_\Gamma(T)$ be a bordered-sutured manifold corresponding to $(Y,T)$ as in Construction~\ref{con:tang-to-BS}. Let $D(Y,T)=-\mathcal{Y}_\Gamma(T)\cup\mathcal{Y}_\Gamma(T)$ be the \emph{double} of $\mathcal{Y}_{\Gamma}(T)$. (See also the beginning of Section~\ref{sec:tangles}.) Then
  \[
    \Mor^{\Alg(\PMC)}\bigl(\lsup{\Alg(\PMC)}\BSD(\mathcal{Y}_{\Gamma}(T)),\lsup{\Alg(\PMC)}\btau_{\Alg(\PMC)}\DT\lsup{\Alg(\PMC)}\BSD(\mathcal{Y}_{\Gamma}(T))\bigr)\simeq \SFC(D(Y,T)).
  \]
\end{corollary}
\begin{proof}
  The negative twisting slice introduces a negative half boundary Dehn twist at each endpoint of the tangle, and hence a negative Dehn twist on each component of the doubled tangle. The bimodule $\btau$ undoes these Dehn twists.
\end{proof}

\subsubsection{Twisted coefficients}
Given a bordered $3$-manifold $(Y,\phi\co F(\PMC)\to\bdy Y)$, there are \emph{totally twisted bordered Floer modules} $\tCFAa(Y,\phi)_{\Alg(\PMC)}$ and $\lsup{\Alg(-\PMC)}\tCFDa(Y,\phi)$, which are equipped with free actions of $\Ftwo[H_2(Y,\bdy Y)]$ commuting with the structure maps over $\Alg(\PMC)$ and $\Alg(-\PMC)$, respectively~\cite[Sections 6.4 and 7.4]{LOT1}. So, given an $\Ftwo[H_2(Y,\bdy Y)]$-module $M$, we can extend scalars and define
\begin{align*}
  \tCFAa(Y,\phi;M)&=\tCFAa(Y,\phi)\otimes_{\Ftwo[H_2(Y,\bdy Y)]}M\\
  \tCFDa(Y,\phi;M)&=\tCFDa(Y,\phi)\otimes_{\Ftwo[H_2(Y,\bdy Y)]}M.
\end{align*}
Then, $\tCFAa(Y,\phi;M)$ is an $\Ainf$-module over the \dg algebra
$\Alg(\PMC)\otimes_{\Ftwo}\Ftwo[H_2(Y,\bdy Y)]$ over the ground ring
$\Idem(\PMC)\otimes_{\Ftwo}\Ftwo[H_2(Y,\bdy Y)]$, and
$\tCFDa(Y,\phi;M)$ is a type $D$ structure over $\Alg(-\PMC)$ with an
action of $\Ftwo[H_2(Y,\bdy Y)]$ commuting with the type $D$ operation
$\delta^1$ (a very special case of a type \DA\ bimodule).

Suppose that $Y_0$ and $Y_1$ are bordered $3$-manifolds, $Y=Y_0\cup_\bdy Y_1$, and $F=\bdy Y_0\subset Y$. Then there is an abelian group homomorphism $H_2(Y)\to H_2(Y,F)\cong H_2(Y_0,\bdy Y_0)\oplus H_2(Y_1,\bdy Y_1)$, which induces a ring homomorphism $i\co \Ftwo[H_2(Y)]\to \Ftwo[H_2(Y,F)]$. In particular, via $i$ we can regard $\Ftwo[H_2(Y,F)]$ as an $\Ftwo[H_2(Y)]$-module and so use $\Ftwo[H_2(Y,F)]$ as a coefficient system for $\tCFa(Y)$. 

The following is the totally twisted pairing theorem:
\begin{theorem}\cite[Theorem 9.44]{LOT1}\label{thm:totally-tw-pair}
  Given bordered $3$-manifolds $Y_0$ and $Y_1$ with boundaries
  $F=F(\PMC)$ and $-F(\PMC)$, respectively, there is a homotopy
  equivalence of differential $\Ftwo[H_2(Y,F)]$-modules
  \[
    \tCFAa(Y_0)_{\Alg(\PMC)}\DT\lsup{\Alg(\PMC)}\tCFDa(Y_1)\simeq \tCFa\bigl(Y;\Ftwo[H_2(Y,F)]\bigr).
  \]
\end{theorem}
(The reference refines this slightly to keep track of $\SpinC$-structures.)

From this we can deduce a general twisted-coefficient pairing theorem:
\begin{corollary}\label{cor:twisted-pairing}
  Fix bordered $3$-manifolds $Y_0$ and $Y_1$ with boundaries $F(\PMC)$
  and $-F(\PMC)$. Let $Y=Y_0\cup_\bdy Y_1$ and $F=\bdy Y_0$. Given
  modules $M$ over $\Ftwo[H_2(Y_0,\bdy Y_0)]$ and $N$ over
  $\Ftwo[H_2(Y_1,\bdy Y_1)]$ there is a homotopy equivalence
  \begin{equation}\label{eq:twisted-pairing}
    \tCFAa(Y_0;M)_{\Alg(\PMC)}\DT\lsup{\Alg(\PMC)}\tCFDa(Y_1;N)\simeq \tCFa(Y;M\otimes_{\Ftwo} N).
  \end{equation}
  of differential modules over
  $\Ftwo[H_2(Y,F)]=\Ftwo[H_2(Y_0,\bdy Y_0)]\otimes\Ftwo[H_2(Y_1,\bdy Y_1)]$.
\end{corollary}
\begin{proof}
  The left side of Equation~\eqref{eq:twisted-pairing} is
  \begin{align*}
    \bigl(\tCFAa(Y_0)&\otimes_{\Ftwo[H_2(Y_0,\bdy Y_0)]} M\bigr)\DT\bigl(\tCFDa(Y_1)\otimes_{\Ftwo[H_2(Y_1,\bdy Y_1)]}N\bigr)\\
    &=\bigl(\tCFAa(Y_0)\DT \tCFDa(Y_1)\bigr)\otimes_{\Ftwo[H_2(Y_0,\bdy Y_0)]\otimes\Ftwo[H_2(Y_1,\bdy Y_1)]}(M\otimes_{\Ftwo} N)
  \end{align*}
  while the right side is
  \begin{align*}
    \tCFa\bigl(Y;&\Ftwo[H_2(Y)]\bigr)\otimes_{\Ftwo[H_2(Y)]}(M\otimes_{\Ftwo} N)\\
                 &=\tCFa\bigl(Y;\Ftwo[H_2(Y)]\bigr)\otimes_{\Ftwo[H_2(Y)]}\Ftwo[H_2(Y,F)]\otimes_{\Ftwo[H_2(Y,F)]}(M\otimes_{\Ftwo} N)\\
                 &=\tCFa\bigl(Y;\Ftwo[H_2(Y,F)]\bigr)\otimes_{\Ftwo[H_2(Y,F)]}(M\otimes_{\Ftwo} N).
  \end{align*}
  So, the isomorphism follows from Theorem~\ref{thm:totally-tw-pair}.
\end{proof}

\section{Non-vanishing results}\label{sec:non-vanishing}
Fix a field $\Field$. Except as indicated, in this
section $\Lambda$ denotes the universal Novikov field over $\Field$ and Floer
homology groups have coefficients in $\Field$ or modules over $\Field$.

The goal of this section is to deduce Theorem~\ref{thm:detect-S2} from a non-vanishing theorem of Ni's, specifically:
\begin{theorem}\cite[Theorem 3.6]{Ni:spheres}\label{Ni:non-vanish}
  Suppose $Y$ is a closed, irreducible $3$-manifold and $F$ a closed surface
  in $Y$ such that $F$ is Thurston-norm minimizing and no subsurface of $F$ is nullhomologous.
  Then there exists a nonempty open set $U\subset H_1(Y;\RR)$ so that for any $\omega\in U$,
  \[
    \tHF^+(Y,\OneHalf x(F);\Lambda_\omega)\neq 0 \qquad\text{and}\qquad \tHFa(Y,\OneHalf x(F);\Lambda_\omega)\neq 0.
  \]
  (Here, $x([F])$ is the Thurston norm of the homology class of $F$,
  and the notation $\OneHalf x(F)$ indicates a certain set of
  $\SpinC$-structures.)
\end{theorem}
(Ni proves Theorem~\ref{Ni:non-vanish} for $\Field=\RR$, but the result for general ground fields $\Field$ is then immediate from the universal coefficients theorem.)

\begin{corollary}\label{cor:Ni-nonvanish}
Let $Y$ be a closed, orientable $3$-manifold with no $S^1\times S^2$ connected summands. Then, there exists a non-empty, open set of closed $2$-forms $U\subset \Omega^2(Y)$ such that for any $\omega\in U$,
$$\tHFa(Y;\Lambda_{\omega})\neq 0.$$   
\end{corollary}
\begin{proof}
  Write $Y=Y_1\#\cdots\# Y_n$ where each $Y_i$ is irreducible. By the
  K\"unneth theorem (Lemma~\ref{lem:Novikov-Kunneth}), it suffices to
  prove the result for each $Y_i$. If $b_1(Y_i)=0$ then
  $\tHFa(Y_i;\Lambda_{\omega_i})=\HFa(Y_i)\otimes_{\Field} \Lambda$,
  so $\chi(\tHFa(Y_i;\Lambda_{\omega_i}))=|H_1(Y_i)|\neq
  0$~\cite[Proposition 5.1]{OS04:HolDiskProperties}. If $b_1(Y_i)>0$
  then choose a Thurston-norm minimizing, connected surface $F$ in
  $Y_i$ which represents some non-zero class in $H_2(Y_i)$. 
  Theorem~\ref{Ni:non-vanish} implies that $0\neq \tHFa(Y_i,\OneHalf x(F);\Lambda_\omega)\subset \tHFa(Y_i;\Lambda_\omega)$, as desired.
\end{proof}

\begin{proof}[Proof of Theorem~\ref{thm:detect-S2}]
The ``if'' direction is essentially the same as an argument of Ni's~\cite[Lemma 2.1]{Ni09:FiberedMfld}. Specifically,
suppose $Y$ contains a $2$-sphere $S$ with $\int_{S}\omega\neq 0$. It
follows that $S$ is homologically essential, hence
nonseparating. Thus, $Y$ decomposes as $Y=Y_0\#(S^1\times S^2)$ such
that $S=\{p\}\times S^2$ for a point $p$ in $S^1$. (The submanifold
$(S^1\times S^2)\setminus B^3$ of $Y$ is a neighborhood of $S$ union
an arc connecting the two sides of $S$.) Let
$\omega_0\in\Omega^2(Y_0)$ and $\omega'\in\Omega^2(S^1\times S^2)$ be
closed $2$-forms such that $[\omega_0]\#[\omega']=[\omega]$. Since
$\int_{S}\omega\neq 0$, then $\int_{S}\omega'\neq 0$ and so, by direct
calculation, $\tHFa(S^1\times S^2;\Lambda_{\omega'})=0$.  So, the claim follows from the K\"unneth theorem, Lemma~\ref{lem:Novikov-Kunneth}.

For the ``only if'' direction, 
split $Y=Y_0\#^k(S^1\times S^2)$ where $Y_0$ has no homologically essential embedded $S^2$. If $\int_{S^2}\omega\neq 0$ for the generator of the second homology of some summand then we are done. Otherwise, let $\omega_0\in\Omega^2(Y_0)$ denotes the restriction of $\omega$.
Since $\tHFa(S^1\times S^2;\Lambda_{\omega_0})\neq 0$ if $\int_{S^2}\omega_0=0$ (by direct computation),
by the K\"unneth theorem, it suffices to show that $\tHFa(Y_0;\Lambda_{\omega_0})\neq 0$. 
By Corollary~\ref{cor:Ni-nonvanish}, there is an open set of closed $2$-forms $U\subset\Omega^2(Y_0)$ such that for any $\alpha\in U$ we have $\tHFa(Y_0;\Lambda_{\alpha})\neq 0$. In particular, there is a generic closed $2$-form $\tilde{\omega}\in\Omega^2(Y_0)$ such that $\tHFa(Y_0;\Lambda_{\tilde{\omega}})\neq 0$. It then follows from Corollary~\ref{HFdim-Novikov} that $\tHFa(Y_0;\Lambda_{\omega_0})\neq 0$.
\end{proof}

\begin{proof}[Proof of Theorem~\ref{thm:HF-nonvanish}]
Split $Y=Y_1\#\cdots\#Y_n$ such that every $Y_i$ is irreducible. By the connected sum formula 
\[\HFa(Y)=\HFa(Y_1)\otimes \cdots\otimes \HFa(Y_n).\]
So, it is enough to prove the theorem for an irreducible closed $3$-manifold.

If $Y$ is a rational homology sphere then the result follows from the fact
that $\chi(\HFa(Y))\neq 0$~\cite[Proposition
5.1]{OS04:HolDiskProperties} and so $\HFa(Y)\neq 0$.

If $b_1(Y)>0$, then let $F$ be a Thurston-norm minimizing
surface in $Y$ such that $[F]\neq 0$ in $H_2(Y)$. Then, $F$ is taut
and a result of Hedden-Ni~\cite[Theorem 2.2]{HeddenNi10:small} implies that $\HFa(Y,[F],\OneHalf x(F))\otimes \QQ\neq 0$ and thus $\HFa(Y)\neq 0$.
\end{proof}

Theorem~\ref{thm:HF-nonvanish} has an extension to the sutured Floer
homology of link complements. Let $L$ be a link in a closed, oriented
$3$-manifold $Y$. We make $Y(L)=Y\setminus\nbd(L)$ into a sutured
manifold by adding two parallel copies, with opposite orientations, of a
homologically essential curve on each component of $\bdy \nbd(L)$ as
sutures. Denote the set of sutures by $\Gamma$. Note that $\Gamma$ is not unique.

Recall that a sutured manifold $(Y,\Gamma)$ is {\it
  irreducible} if every embedded $S^2$ in $Y$ bounds a
$3$-ball. Otherwise, $Y$ is \emph{reducible}.

\begin{theorem}\label{thm:SFH-nonvanishing}
There is a suture $\gamma\in\Gamma$ which bounds an embedded disk in $Y(L)$ if and only if $\SFH(Y(L),\Gamma)=0$.
\end{theorem}
\begin{proof}
Consider a decomposition of $(Y(L),\Gamma)$ as
\[(Y(L),\Gamma)=(Y_0,\Gamma_0)\#\cdots\#(Y_k,\Gamma_k)\#Y_{k+1}\#\cdots\#Y_n\#\left(\#^m(S^1\times S^2)\right)\]
so that:
\begin{itemize}
\item for any $k+1\le i\le n$, $Y_i$ is an irreducible closed $3$-manifold,
\item for any $0\le i\le k$, $(Y_i,\Gamma_i)$ is an irreducible
  sutured manifold where $\Gamma_i=\Gamma\cap Y_i$. 
\end{itemize}
Denote the sublink of $L$ whose corresponding set of sutures is $\Gamma_i$ by $L_i$.

It follows from the connected sum formula for sutured Floer homology ~\cite[Proposition 9.15]{Juhasz06:Sutured} that 
\[
  \SFH(Y(L),\Gamma)=\left(\bigotimes_{i=0}^{k}\SFH(Y_i,\Gamma_i)\right)\otimes\left(\bigotimes_{j=k+1}^{n} \HFa(Y_{j})\right)\otimes (\Field^{2})^{\otimes(m+k)}.
\]
Thus, Theorem~\ref{thm:HF-nonvanish} implies that $\SFH(Y(L),\Gamma)=0$ if and only if $\SFH(Y_i,\Gamma_i)=0$ for some $0\le i\le k$.

Suppose $\SFH(Y_i,\Gamma_i)=0$ for some $i$. Theorem~\ref{thm:nonvanishSFH} implies that $(Y_i,\Gamma_i)$ is not taut. All components of $R(\Gamma_i)$ are annuli, hence $R_+(\Gamma_i)$ and $R_-(\Gamma_i)$ are Thurston-norm minimizing, so 
 there is a suture $\gamma\in\Gamma_i$ which bounds a properly embedded disk in $Y_i$.

Conversely, if a suture $\gamma\in\Gamma_i$ bounds a disk $D$ in
$Y(L)$, we inductively construct a disk $D_i$ in $Y_i$ whose boundary
is $\gamma$. Perturb $D$ so it intersects the connected sum
$2$-spheres transversely, and denote the intersection by $Z$. If
$Z=\emptyset$, then $D$ lies in $Y_i$ and set $D_i=D$. Otherwise,
choose an innermost circle $Z_0$ in $Z\subset D$ and let $D'$ be the
disk in $D$ bounded by $Z_0$. For some $0\le j\le n$, the disk $D'$ is
in $Y_j$. Let $S$ be the connected sum sphere that contains
$Z_0$. Since $Y_j$ is irreducible, for one of the disks bounding $Z_0$
in $S$, denoted by $D''$, $D'\cup D''$ bounds a $3$-ball in
$Y_j\cap Y(L)$. Replace $S$ with an isotopic copy of $(S\setminus D'')\cup D'$, to get a connected sum decomposition of $(Y(L),\Gamma)$ into irreducible closed or sutured manifolds, such that the intersection of $D$ with the connected sum spheres has fewer components. Continue this process until $Z$ is empty, and the result is a disk $D_i$ in $Y_i$ with $\bdy D_i=\gamma$. Thus, the irreducible sutured manifold $(Y_i,\Gamma_i)$ is not taut. It follows from Theorem~\ref{thm:vanishSFH} that $\SFH(Y_i,\Gamma_i)=0$, which implies $\SFH(Y(L),\Gamma)=0$.
\end{proof}

\section{Bordered Heegaard Floer homology detects incompressible surfaces}\label{sec:bordered-detects}
The goal of this section is to prove Theorem~\ref{thm:detect-incompress}. We
start by defining the bimodule $\lsup{\Alg(\PMC)}\bLambda(\PMC)_{\Alg(\PMC)}$.
Fix a pointed matched circle $\PMC$ for a surface of genus $k$. Fix also, once and
for all, $4k-1$ real numbers $\lambda_1,\dots,\lambda_{4k-1}$ which are linearly
independent over $\QQ$; for instance,
$\lambda_1=\sqrt{2},\lambda_2=\sqrt{3},\lambda_3=\sqrt{5},\cdots$. There is an
induced injection $\psi\co \ZZ^{4k-1}\to \RR$ which sends the $i\th$ standard
basis vector to $\lambda_i$.

Recall that a strand diagram $a\in\Alg(\PMC)$ has a \emph{support}
$[a]\in \ZZ^{4k-1}$.

\begin{definition}
  As a $\Lambda\otimes_{\Ftwo}\Idem(\PMC)$-module, define
  $\bLambda(\PMC)=\Lambda\otimes_{\Ftwo}\Idem(\PMC)$. The structure maps
  $\delta^1_n$ on $\bLambda(\PMC)$ are defined to vanish if $n\neq 2$, and
  $\delta^1_2\co \bLambda(\PMC)\otimes \Alg(\PMC)\to \Alg(\PMC)\otimes
  \bLambda(\PMC)$ is a homomorphism of $\Lambda$-vector spaces. So, it
  only remains to define $\delta^1(i\otimes a)$ for $i$ a basic idempotent
  (viewed as a generator of $\bLambda(\PMC)$) and $a$ a
  strand diagram. Define
  \[
    \delta^1_2(i\otimes a)=
    \begin{cases}
      a\otimes T^{\psi([a])}j& \text{if }ia\neq 0,\text{ where }aj\neq 0\\
      0 & \text{otherwise}.
    \end{cases}
  \]
\end{definition}
It is immediate from the definitions that
$\lsup{\Alg(\PMC)}\bLambda(\PMC)_{\Alg(\PMC)}$ is a type \DA\ bimodule over
$\Alg(\PMC)$ and, in fact, over $\Alg(\PMC)\otimes_{\Ftwo}\Lambda$.

The homomorphism $\psi$ also induces an element
$[\psi]\in H^2(Y,\bdy Y;\RR)=\Hom(H_2(Y,\bdy Y),\RR)$ as follows. The
identification of $\bdy Y$ with $F(\PMC)$ gives an embedding
$\iota\co H_1(\bdy Y)\into \ZZ^{4k-1}$. So, given an element
$A\in H_2(Y,\bdy Y)$ there is a corresponding element
$\iota(\bdy A)\in \ZZ^{4k-1}$. Define $[\psi](A)=\psi(\iota(\bdy A))\in \RR$.

Next, let $(-Y)\cup_\bdy Y$ be the double of $Y$ along its boundary. Via the homomorphism
\[
  H^2(Y,\bdy Y;\RR)\cong H^2((-Y)\cup_\bdy Y,-Y;\RR)\to H^2((-Y)\cup_\bdy Y;\RR)
\]
we can view $[\psi]$ as an element of $H^2((-Y)\cup_\bdy Y;\RR)$ and hence there
is a corresponding Novikov coefficient system $\Lambda_{[\psi]}$ on
$(-Y)\cup_\bdy Y$. Note that $[\psi]$ vanishes on $A\in H_2((-Y)\cup_\bdy Y)$ if
and only if $A\cdot [\bdy Y]=0\in H_1(\bdy Y)$. (This uses the fact that the
$\lambda_i$ are linearly independent over $\QQ$.) This intersection product agrees with the second map in the Mayer-Vietoris sequence
\[
  H_2(-Y)\oplus H_2(Y)\to H_2((-Y)\cup_\bdy Y)\to H_1(\bdy Y),
\]
so $[\psi](A)=0$ if and only if $A$ comes from $H_2(-Y)\oplus H_2(Y)$. In other words, $[\psi]$ is in the image of the map $H^1(\bdy Y)\to H^2((-Y)\cup_\bdy Y)$
from the Mayer-Vietoris sequence.

Finally, given a bordered $3$-manifold $(Y,\phi\co F(-\PMC)\to \bdy Y)$, the composition
\[
  H_2(Y,\bdy Y)\stackrel{\bdy}{\longrightarrow} H_1(\bdy Y)
  \stackrel{\phi_*^{-1}}{\longrightarrow}H_1(F(-\PMC))
  \stackrel{[\psi]}{\longrightarrow}\RR
\]
makes the universal Novikov field $\Lambda$ into an algebra over
$\Ftwo[H_2(Y,\bdy Y)]$. Call this coefficient system $\Lambda_{\psi}$
and the resulting twisted bordered module
$\tCFDa(Y;\Lambda_{\psi})$. It is straightforward from the definitions that
\begin{equation}
  \label{eq:tCFD-is-box}
  \lsup{\Alg(\PMC)}\tCFDa(Y;\Lambda_{\psi})\cong \lsup{\Alg(\PMC)}\bLambda(\PMC)_{\Alg(\PMC)}\DT\lsup{\Alg(\PMC)}\CFDa(Y).
\end{equation}

\begin{lemma}\label{lem:tw-mor-cx}
  Fix a bordered $3$-manifold $(Y,\phi\co F(-\PMC)\to \bdy Y)$. Then there is a quasi-isomorphism
  \[
    \Mor_{\Alg(\PMC)}(\lsup{\Alg(\PMC)}\CFDa(Y,\phi),\lsup{\Alg(\PMC)}\bLambda(\PMC)_{\Alg(\PMC)}\DT\lsup{\Alg(\PMC)}\CFDa(Y,\phi))
    \simeq \tCFa((-Y)\cup_\bdy Y;\Lambda_{[\psi]}).
  \]
\end{lemma}
\begin{proof}
  Since $\CFAa(-Y,\phi)$ is dual to $\CFDa(Y,\phi)$~\cite[Theorem
  2]{LOTHomPair}, the statement is equivalent to
  \[
    \CFAa(-Y,\phi)_{\Alg(\PMC)}\DT\lsup{\Alg(\PMC)}\bLambda(\PMC)_{\Alg(\PMC)}\DT\lsup{\Alg(\PMC)}\CFDa(Y,\phi))\simeq \tCFa((-Y)\cup_\bdy Y;\Lambda_{[\psi]}).
  \]
  So, the result follows from Equation~\eqref{eq:tCFD-is-box} and
  the twisted pairing theorem, Corollary~\ref{cor:twisted-pairing}.
\end{proof}

Call an element of $H^1(\bdy Y;\RR)$ \emph{generic} if the induced map
$H_1(\bdy Y)\to \RR$ is an embedding. The element
$[\psi]\in H^2((-Y)\cup_\bdy Y)$ is the image of a generic element of
$H^1(\bdy Y;\RR)$. Since we continue to think of real cohomology classes as represented by differential forms, we will use integral signs to represent the pairing between homology and cohomology. The following is a variant on a well-known lemma in 3-manifold topology (cf.~\cite[Lemma 3.4]{GL15:OB-prime},~\cite[Section 7]{Haken68:surfaces}):
\begin{lemma}\label{lem:disk-to-sphere}
  Let $Y$ be a $3$-manifold with connected boundary and $\omega$ a generic
  element of $H^1(\bdy Y;\RR)$. Let $\omega'$ be the induced
  element of $H^2((-Y)\cup_{\bdy} Y;\RR)$.  There is an embedded sphere $S$ in
  $(-Y)\cup_\bdy Y$ with $\int_S\omega'\neq 0$ if and only if there is
  an embedded disk in $(Y,\bdy Y)$ with $\int_{\bdy D}\omega\neq 0$.
\end{lemma}
Since $\omega$ is generic, the condition that
$\int_{\bdy D}\omega\neq 0$ is equivalent to
$[\bdy D]\neq 0\in H_1(\bdy Y)$.
\begin{proof}
  If there is an embedded disk in $(Y,\bdy Y)$ with
  $\int_{\bdy D}\omega\neq 0$ then, letting
  $S=(-D)\cup_\bdy D\subset (-Y\cup_\bdy Y)$, we have
  $\int_S\omega'=\int_{\bdy D}\omega\neq 0$.

  For the converse, suppose there is an embedded sphere $S$ in
  $(-Y)\cup_\bdy Y$ with $\int_S\omega\neq 0$. Perturbing $S$
  slightly, we may assume that $S$ is transverse to $\bdy Y$, so
  $Z=S\cap\bdy Y$ is an embedded $1$-manifold in $S$. If $Z$ has a
  single component then we can take $D$ to be either hemisphere of
  $S\setminus Z$. So, we will inductively reduce the number of
  components of $Z$, while preserving the properties that:
  \begin{itemize}
  \item $\int_S\omega'\neq 0$ and
  \item there is an open neighborhood $U\subset (-Y)\cup_\bdy Y$ of
    $\bdy Y$ so that $S\cap U$ is an embedded submanifold. That is,
    $S$ is \emph{embedded near $\bdy Y$}.
  \end{itemize}
  At the intermediate steps, $S$ may not itself be embedded. Our goal
  is to find a disk $D$ in $Y$ or $-Y$ which is embedded near $\bdy Y$
  and so that $\int_{\bdy D}\omega=\int_D\omega'\neq 0$. Dehn's lemma 
  then gives an embedded disk $D'$ in $Y$ with the same
  boundary, proving the result.

  So, choose an innermost circle $Z_0$ in $Z\subset D$ and let $D_0$ be the
  disk in $S$ bounded by $Z_0$. Without loss of generality, suppose
  $D_0\subset Y$. Let $S_0=(-D_0)\cup_{Z_0}D_0$ be the double of
  $D_0$, which is a sphere in $(-Y)\cup_\bdy Y$, embedded near
  $\bdy Y$. Since $\int_{\bdy D_0}\omega=\int_{S_0}\omega'$, if
  $\int_{S_0}\omega'\neq 0$ then we are done. If $\int_{S_0}\omega'=0$ then
  $S'=(S\setminus D_0)\cup(-D_0)$ is another sphere in
  $(-Y)\cup_\bdy Y$, and 
  $
    \int_{S'}\omega'=\int_S\omega'-\int_{S_0}\omega'=\int_S\omega'\neq 0.
  $
  We can isotope $S'$ to intersect $\bdy Y$ in
  $Z\setminus Z_0$. In particular, $S'$ is still embedded near
  $\bdy Y$. So, by induction on the number of components of $Z$, we
  are done.
\end{proof}

\begin{proof}[Proof of Theorem~\ref{thm:detect-incompress}]
  By Lemma~\ref{lem:tw-mor-cx}, we want to show that
  $\tHFa((-Y)\cup_\bdy Y;\Lambda_{[\psi]})=0$ if and only if $\bdy Y$
  has a homologically essential compressing disk. By
  Lemma~\ref{lem:disk-to-sphere}, $\bdy Y$ has a homologically
  essential compressing disk if and only if $(-Y)\cup_\bdy Y$ contains
  an embedded sphere $S$ with $\int_{S}[\psi]\neq 0$. By
  Theorem~\ref{thm:detect-S2}, this condition is equivalent to
  $\tHFa((-Y)\cup_\bdy Y;\Lambda_{[\psi]})=0$.
\end{proof}

\begin{remark}\label{rem:weak-compress}
  Call a compressing disk $D$ in $Y$ \emph{weakly homologically essential} if $[D]\neq 0\in H_2(Y,\bdy Y)$. If $Y$ has no $S^2\times S^1$ summands then the totally twisted bordered Floer module $\tCFDa(Y,\bdy Y)$ detects whether $\bdy Y$ has a weakly homologically essential compressing disk, as follows. Choose an injection $\omega\co  H_2(Y,\bdy Y)\to \RR$, making the universal Novikov field into a module $\Lambda_\omega$ over $\Ftwo[H_2(Y,\bdy Y)]$. From $\tCFDa(Y,\bdy Y)$ we can construct $\tCFDa(Y,\bdy Y;\Lambda_\omega)$. Then
  \[
    H_*\Mor(\CFDa(Y,\bdy Y),\tCFDa(Y,\bdy Y;\Lambda_\omega))=0
  \]
  if and only if $\bdy Y$ has a weakly homologically essential compressing disk. To see this, imitate the proof of Lemma~\ref{lem:disk-to-sphere} to show that $Y$ has a weakly homologically essential compressing disk if and only if the double of $Y$ has a homologically essential $2$-sphere $S$ such that $\omega$ evaluates nontrivially on $[S\cap Y]\in H_2(Y,\bdy Y)$. Then apply Theorem~\ref{thm:detect-S2}.
\end{remark}

\section{Multiple boundary components}\label{sec:multiple-bdy}
Theorem~\ref{thm:detect-incompress} has generalizations to
manifolds with multiple boundary components. For notational simplicity
we will focus on the case of manifolds with two boundary components;
there are no essential differences in the general case.

\begin{figure}
  \centering
  \includegraphics[scale=.8]{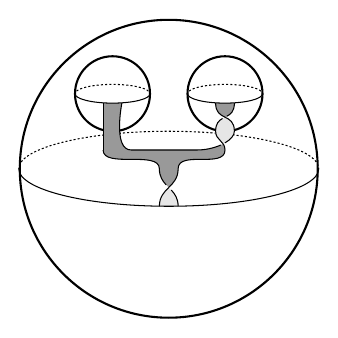}
  \caption{\textbf{Framed prongs.} This is an analogue of an arced cobordism for $Y=\bar{B}^3\setminus (B^3\amalg B^3)$. The embedded disk which plays the role of the framed arc is shaded.}
  \label{fig:prongs}
\end{figure}

So, consider a compact $3$-manifold $Y$ with boundary components
$\bdy_LY$ and $\bdy_RY$. Choosing parametrizations
$\phi_L\co F(\PMC_L)\to\bdy_LY$ and $\phi_R\co F(\PMC_R)\to\bdy_R Y$
and a framed arc $\gamma$ in $Y$ from $\phi_L(z)$ to $\phi_R(z)$ makes
$Y$ into an \emph{arced cobordism}. (Thinking of the framed arc as a
ribbon, the generalization to the case of more than two boundary
components is to choose an embedding of a closed disk $\bD^2$ so that
one arc on $\bdy\bD^2$ maps to each component of $\bdy Y$. See
Figure~\ref{fig:prongs}, as well as~\cite[Remark
5.7]{LOT2} and~\cite[Section 2]{Hanselman16:graph}.)  There is a corresponding left-left
type \DD\ structure
$\lsup{\Alg(-\PMC_L),\Alg(-\PMC_R)}\CFDDa(Y)$~\cite[Section 6]{LOT2}. Equivalently, we can
think of $\CFDDa(Y)$ as a left type $D$ structure over
$\Alg(-\PMC_L)\otimes_{\Ftwo}\Alg(-\PMC_R)$.

The endomorphism complex of
$\lsup{\Alg(-\PMC_L)\otimes_{\Ftwo}\Alg(-\PMC_R)}\CFDDa(Y)$ does not
compute the Heegaard Floer homology of the double of $Y$. Rather, if
we let $K=(-\gamma)\cup_{\bdy}\gamma\subset (-Y)\cup_{\bdy} Y$ then
\begin{multline}\label{eq:self-pair-1}
  \Mor^{\Alg(-\PMC_L)\otimes_{\Ftwo}\Alg(-\PMC_R)}(\lsup{\Alg(-\PMC_L)\otimes_{\Ftwo}\Alg(-\PMC_R)}\CFDDa(Y),\lsup{\Alg(-\PMC_L)\otimes_{\Ftwo}\Alg(-\PMC_R)}\CFDDa(Y))\\
  \simeq
  \CFKa\bigl([(-Y)\cup_{\bdy} Y]_{-1}(K),K'\bigr)
\end{multline}
where $K'$ is the core of the surgery torus in the result of $-1$-surgery on $K$~\cite[Theorem 6]{LOTHomPair}.

To obtain the Heegaard Floer homology of $(-Y)\cup_{\bdy} Y$ involves an
extra bimodule, denoted $\BSDA(\TC(\PMC))$ in earlier
work~\cite{LT:hoch-loc} (see also~\cite{Hanselman16:graph}). We can think
of an arced cobordism $(Y,\phi_L,\phi_R,\gamma)$ as a bordered-sutured
manifold $Y(\gamma)$ by deleting a neighborhood of $\gamma$ and
placing two longitudinal sutures along the result (using the
framing). (In the multiple boundary component case, the sutures are
the edges of the ribbon.) In this setting, Formula~\eqref{eq:self-pair-1} follows from
Theorem~\ref{thm:BS-hom-pair}; the surgery comes from the two twisting
slices. For a pointed matched circle $\PMC$, if we think of $F(\PMC)$ as a surface with one boundary
component (i.e., think of $\PMC$ as an arc diagram) then the
bordered-sutured cobordism $\TC(\PMC)$ is obtained from
$[0,1]\times F(\PMC)$ by attaching a $3$-dimensional $2$-handle to
$\{1/2\}\times(\bdy F(\PMC))$. The bordered boundary of $\TC(\PMC)$ is
$\{0,1\}\times F(\PMC)$. The remainder of $\bdy\TC(\PMC)$ consists of
two $2$-disks. There is a single suture arc on each of these $2$-disks, so $R_+$ and $R_-$ each consists of two disks, one on each component of $\bdy\TC(\PMC)$.

It follows that
\[
  \bigl(-Y(\gamma)\bigr)\cup_{F(-\PMC_L)\amalg F(-\PMC_R)} \bigl(\TC(-\PMC_R)\cup_{F(-\PMC_R)}Y(\gamma)\bigr)\cong
  (-Y\cup_{\bdy}Y)\setminus \bD^3,
\]
as sutured manifolds. 
Thus, the following result is the natural generalization of Theorem~\ref{thm:detect-incompress}:
\begin{theorem}\label{thm:detect-incompress-2-bdy}
  Fix a compact, connected, oriented 3-manifold $Y$ with two boundary
  components, $\bdy_LY$ and $\bdy_RY$. Make $Y$ into an arced
  cobordism with boundary $F(-\PMC_L)$ and $F(-\PMC_R)$ for some pointed
  matched circles $\PMC_L$ and $\PMC_R$. Then
  \begin{equation}\label{eq:detect-incompress-2-bdy}
    H_*\Mor^{\Alg(\PMC_L)\otimes\Alg(\PMC_R)}\bigl(\CFDDa(Y),\bLambda(\PMC_R)\DT_{\Alg(\PMC_R)}\BSDA(\TC(\PMC_R))\DT_{\Alg(\PMC_R)}\CFDDa(Y)\bigr)=0
  \end{equation}
  if and only if $\bdy_RY$ is compressible.
\end{theorem}
(The generalization of Theorem~\ref{thm:detect-incompress-2-bdy} to
$n>2$ boundary components involves $n-1$ copies of $\TC$.)
\begin{proof}
  From the discussion above, the left side of
  Formula~\eqref{eq:detect-incompress-2-bdy} computes
  $\tHFa((-Y)\cup_\bdy Y;\Lambda_\omega)$ where
  $\omega\in H^2((-Y)\cup_{\bdy} Y;\RR)$ is the image of
  $[\psi]\in H^1(\bdy_RY;\RR)$ under the composition
  \begin{align*}
    H^1(\bdy_R Y;\RR)&\stackrel{(\Id,0)}{\longrightarrow}H^1(\bdy_R Y;\RR)\oplus H^1(\bdy_L Y;\RR)\cong H^1(\bdy Y;\RR)\\
    &\longrightarrow H^2(Y,\bdy Y;\RR)\cong H^2((-Y)\cup_\bdy Y,-Y;\RR)
    \longrightarrow H^2((-Y)\cup_\bdy Y;\RR)
  \end{align*}
  (see Section~\ref{sec:bordered-detects}). So, the result follows
  from an analogue of Lemma~\ref{lem:disk-to-sphere}:

  \begin{lemma}\label{lem:disk-to-sphere-2-bdy}
    Let $Y$ be a compact $3$-manifold with boundary
    $\bdy_LY\amalg\bdy_R Y$ and $\omega$ a generic element of
    $H^1(\bdy_R Y;\RR)$. Let $\omega'$ be the induced element of
    $H^2((-Y)\cup_\bdy Y;\RR)$.  There is an embedded sphere $S$ in
    $(-Y)\cup_\bdy Y$ with $\int_S\omega'\neq 0$ if and only if there
    is an embedded disk $D$ in $(Y,\bdy_R Y)$ with
    $\int_{\bdy D}\omega\neq 0$.
  \end{lemma}
  The proof of Lemma~\ref{lem:disk-to-sphere-2-bdy} is similar to the
  proof of Lemma~\ref{lem:disk-to-sphere} and is left to the reader.
\end{proof}

Note in particular that one can use Theorem~\ref{thm:detect-incompress-2-bdy} to test whether Seifert surfaces are incompressible (from either side).

\section{Partly boundary parallel tangles}\label{sec:tangles}
In this section, we prove Theorem~\ref{thm:detect-tang}. 

Suppose $Y$ is an oriented $3$-manifold and $T$ is a tangle in
$Y$. (Recall our convention that all components of all tangles are assumed to be
nullhomologous. So, in particular, their meridians are homologically essential in $Y(T)=Y\setminus\nbd(T)$.) An embedded disk $(D, \bdy D)$ in $(Y\setminus \nbd(T),\bdy
Y)$ is called \emph{essential}, if $D$ is not isotopic relative
boundary, in $Y\setminus \nbd(T)$,  to an embedded disk in $\bdy Y$. The
tangle $(Y,T)$ is called \emph{irreducible} if it has no essential
disk, and otherwise, it is called \emph{reducible}. Let $L$ denote the
nullhomologous link specified by doubling $T$ i.e., $L=-T\cup_{\bdy T} T$ in the closed $3$-manifold $X=-Y\cup_{\bdy Y} Y$. 

Corresponding to $(Y,T)$, there is a bordered-sutured manifold
$\mathcal{Y}_\Gamma(T)=(Y(T),\Gamma)$ with bordered and sutured
boundaries as described in Construction \ref{con:tang-to-BS}. Further,
for any interval component $T_i$ of $T$, we can modify $\Gamma$ to get
a new set $\Gamma_i$ of sutures and a bordered-sutured manifold
$\mathcal{Y}_{\Gamma_i}(T)=(Y(T),\Gamma_i)$.  Let $F(-\PMC)$ be the
bordered boundary of $\mathcal{Y}_{\Gamma}(T)$ (and
$\mathcal{Y}_{\Gamma_i}(T)$). The double
$D(Y,T)=-\mathcal{Y}_\Gamma(T)\cup_{F(\PMC)}\mathcal{Y}_\Gamma(T)$ is
the sutured manifold constructed from $X(L)=X\setminus \nbd(L)$ by
adding a pair of meridional sutures on any closed components of $-T$
or $T$ and a pair of longitudinal sutures on the rest of the
components. Note that despite there being many choices of longitudinal
sutures in $\Gamma$, the sutures $-\Gamma\cup_{\bdy}\Gamma$ of
$D(Y,T)$ are unique: the doubles of the longitudinal sutures are longitudes for
the nullhomologous knot $-T\cup_\bdy T$, and hence are well-defined up
to isotopy.

Moreover, for any choice of interval component $T_i$ of $T$, one may construct a sutured manifold from $X(L)$ by adding a pair of 
longitudinal sutures along $-T_i\cup_\bdy T_i$ and a pair of meridional sutures around the rest of the components. Denote this sutured manifold by $D_i(Y,T)$.

\begin{proof}[Proof of Theorem~\ref{thm:detect-tang}] By Corollary~\ref{cor:doubling} we have: 
  \[
  \SFH(D(Y,T))\cong  H_*\Mor^{\Alg(\PMC)}\left(\lsup{\Alg(\PMC)}\BSD(\mathcal{Y}_\Gamma(T)),\lsup{\Alg(\PMC)}\btau_{\Alg(\PMC)}\DT\lsup{\Alg(\PMC)}\BSD(\mathcal{Y}_\Gamma(T))\right).
  \]
Therefore, 
$$
H_*\Mor_{\Alg(\PMC)}\left(\lsup{\Alg(\PMC)}\BSD(\mathcal{Y}_\Gamma(T)),\lsup{\Alg(\PMC)}\btau_{\Alg(\PMC)}\DT\lsup{\Alg(\PMC)}\BSD(\mathcal{Y}_\Gamma(T))\right)=0.
$$
if and only if $\SFH(D(Y,T))=0$.

Choose a set $D=\{D_1,\dots,D_n\}$ of disjoint essential disks in $(Y,T)$ so that for the decomposition
$$Y\setminus \nbd(D)=Y_0\amalg\cdots\amalg Y_n$$
all pairs $(Y_i,T_i=Y_i\cap T)$ are irreducible. Note that some $T_i$ might be empty. This decomposition corresponds to a decomposition of the form
\[
  D(Y,T)=D(Y_0,T_0)\#\cdots\#D(Y_k,T_k)\#\bigl(\#^{n-k}(S^1\times S^2)\bigr)
  \]

By the connected sum formula, $\SFH(D(Y,T))=0$ if and only if for some $0\le i\le k$
we have $\SFH(D(Y_i,T_i))=0$. Therefore, it is enough to prove the theorem for an irreducible tangle.

Assume $(Y,T)$ is irreducible. It follows from Theorem~\ref{thm:SFH-nonvanishing} that $\SFH(D(Y,T))=0$ if and only if there is a suture $\gamma$ that bounds an embedded disk in $X(L)$.  Since $L$ is nullhomologous,
meridional sutures are homologically essential. So, $\gamma$ is a longitudinal suture and the corresponding component of $L$ bounds a disk in $X\setminus L$. Therefore, $\SFH(D(Y,T))=0$ is equivalent to existence of an interval component $T_0$, such that $-T_0\cup_{\bdy} T_0$
bounds a disk in $X\setminus L$. It remains to show that
$-T_0\cup_{\bdy} T_0$ bounding a disk in $X\setminus L$ is equivalent
to $T_0$ being partly boundary parallel. One direction, that $T_0$
partly boundary parallel implies $-T_0\cup_\bdy T_0$ bounds a disk, is
obvious. For the other, let $D$ denote the corresponding embedded disk
in $X\setminus L$ where $\bdy D=-T_0\cup_{\bdy} T_0$ and $Z=D\cap \bdy
Y$. If $Z$ is a single arc connecting the points in $\bdy T_0$, then
$T$ is partly boundary parallel. If $Z$ has some circle components,
consider an innermost circle $Z_0$ in $Z$ which bounds an empty disk
$D_0\subset D$. Since $(Y,T)$ is irreducible,  we can isotope $D$ so
that after the isotopy it intersects $\bdy Y$ at $Z\setminus
Z_0$. Continue this process until $Z$ has only one component, and we
are done.

For the second part, note that the sutured manifold $m(\mathcal{Y}_{\Gamma_i}(T))\cup\mathcal{TW}_{F(-\PMC),-}\cup\mathcal{Y}_{\Gamma_i}(T)$ is obtained from $D_i(Y,T)$ by a negative Dehn twist on the component corresponding to $-T_i\cup_{\bdy}T_i$. Thus, 
$$
H_*\Mor_{\Alg(\PMC)}\left(\lsup{\Alg(\PMC)}\BSD(\mathcal{Y}_{\Gamma_i}(T)),\lsup{\Alg(\PMC)}\btau^i_{\Alg(\PMC)}\DT\lsup{\Alg(\PMC)}\BSD(\mathcal{Y}_{\Gamma_i}(T))\right)\cong \SFH(D_i(Y,T)),
$$
where $\lsup{\Alg(\PMC)}\btau^i_{\Alg(\PMC)}$ is the DA bimodule
corresponding to performing a Dehn twist on one of the components of $\bdy F(-\PMC)$ specified by $T_i$. On the other hand, since $L$ is nullhomologous, the meridional sutures do not bound disks, so $\SFH(D_i(Y,T))=0$ if and only if the longitudinal suture corresponding to $-T_i\cup_{\bdy} T_i$ bounds a disk in $X(L)$. Equivalently, $-T_i\cup_{\bdy} T_i$ bounds a disk in $X\setminus L$. The rest of the proof is completely similar to the proof of the first part.
\end{proof}

\section{Computationally effective versions}\label{sec:comp}
Finally, we show that
Theorems~\ref{thm:detect-incompress} and~\ref{thm:detect-tang} can be
adapted for machine verification. (This is not
the first algorithm for detecting incompressibility.)
There are three barriers to applying these theorems to
verify incompressibility and partial boundary parallelness:
\begin{enumerate}[label=({\tiny \RainCloud}\arabic*)]
\item\label{item:comp-1} Computing $\lsup{\Alg(\PMC)}\CFDa(Y,\phi)$
  (Theorem~\ref{thm:detect-incompress}) or $\lsup{\Alg(\PMC)}\BSD(\mathcal{Y}_{\Gamma}(T))$
(Theorem~\ref{thm:detect-tang}).
\item\label{item:comp-1.5} In the case of Theorem~\ref{thm:detect-tang}, computing the bimodule $\lsup{\Alg(\PMC)}\btau_{\Alg(\PMC)}$.
\item\label{item:comp-2} Computing the homology of the morphism
  complex.
\end{enumerate}
The difficulty in~\ref{item:comp-1} and~\ref{item:comp-1.5} is that the bordered or bordered-sutured modules are defined by counting pseudoholomorphic curves. The difficulty in~\ref{item:comp-2} is that a general element of the Novikov field $\Lambda$ consists of infinitely much data---the sequence $a_i$ of real numbers---and so is not well-adapted to computer linear algebra.

\subsection{Computing bordered and bordered-sutured modules}
For $\CFDa$,~\ref{item:comp-1} was addressed in earlier work of
Ozsv\'ath, Thurston, and the second author~\cite[Section 8]{LOT4}. The
strategy is to decompose $Y$ as the union of a handlebody $H_0$ and a
compression body $H_1$, glued via a diffeomorphism $\psi$. For
standard parametrizations $\phi_0$ of $\bdy H_0$ and $\phi_1\amalg
\phi'_1$ of $\bdy H_1$ it is straightforward to compute
$\CFAa(H_0,\phi_0)$ and $\CFDAa(H_1,\phi_1\amalg \phi'_1)$, say. One
then factors $\phi_1^{-1}\circ\psi\circ
\phi_0=\psi_1\circ\cdots\circ\psi_n$ as a product of generators of the
mapping class groupoid, called \emph{arcslides}. One also factors
$\phi^{-1}\circ\phi'_1$ into arcslides
$\chi_1\circ\cdots\circ\chi_m$. The bimodules $\CFDDa(\psi_i)$ and
$\CFDDa(\chi_i)$ are complicated but can be described explicitly (cf.\
Section~\ref{sec:arcslides}); this
is the main work. One can then compute the type \DA\ bimodules for
these arcslides using some dualities (cf.\ Section~\ref{sec:identity}). Finally, one has
\begin{multline*}
  \CFDa(Y,\phi)\simeq \CFAa(H_0,\phi_0)\DT\CFDAa(\psi_1)\DT\cdots\DT\CFDAa(\psi_n)\DT\CFDAa(H_1,\phi_1\amalg\phi'_1)\\
  \DT\CFDAa(\chi_1)\DT\cdots\DT\CFDAa(\chi_{m-1})\DT\CFDDa(\chi_m).
\end{multline*}

Though it has not appeared in the literature, a similar strategy works
to compute bordered-sutured modules. There are a few more \emph{basic
  bordered-sutured pieces} that must be computed:
\begin{enumerate}
\item Arcslides between arc diagrams. This is a simple extension of the bordered computation~\cite{LOT4} (see Section~\ref{sec:arcslides}).
\item One- and two-handle attachments to the bordered boundary. The bimodules for $1$- and $2$-handle
  attachments are the same, except for which action corresponds to
  which boundary component. It will be sufficient to consider
  $1$-handles with both feet in the same
  bordered boundary component, and $2$-handles which do not disconnect
  a boundary component.

  The relevant bordered-sutured Heegaard diagrams are shown in
  Figure~\ref{fig:interiorHandle}.
  In the case of 1-handle attachments to a pointed matched circle, the
  bimodule was already computed~\cite[Section 8.1]{LOT4}, and the
  extension to arc diagrams is straightforward (see
  Section~\ref{sec:interior-handle}). (In fact, the pointed matched
  circle case suffices for us.)
\item Attaching a $1$-handle to $R_-$. The Heegaard diagram is shown in Figure~\ref{fig:RminusOneHandle}. Topologically, this cobordism is a product $[0,1]\times F(\PMC)$, but there is a (particular) $1$-handle $H$ in $F(\PMC)$ so that $R_-$ is $\bigl([0,1]\times S_-(F(\PMC))\bigr)\cup \bigl(\{1\}\times H\bigr)$. The set $R_+$ is $[0,1]\times S_+(F(\PMC))$. The bordered-sutured invariant can be deduced from the bordered-sutured invariant of the identity cobordism (See Section~\ref{sec:R-handles}).
\item Attaching a $1$-handle to $R_+$. The Heegaard diagram is shown in Figure~\ref{fig:RplusOneHandle}. This is similar to the previous case, except that $\{1\}\times H$ becomes part of $R_+$ instead of $R_-$. Again, the bordered-sutured invariant can be deduced from the bordered-sutured invariant of the identity cobordism (See Section~\ref{sec:R-handles}).
\item An $R_+$-cup or cap. The Heegaard diagram is shown in Figure~\ref{fig:cupInRminus}. Topologically, this cobordism is a product $[0,1]\times F(\PMC)$. All but one component of $R_-$ (respectively $R_+$) in the sutured boundary is of the form $[0,1]\times S_-(F(\PMC))$ (respectively $[0,1]\times S_+(F(\PMC))$). For a cup, one component of $R_-$ is a hexagon with three sides components of $\Gamma$, two sides on $\{1\}\times S_-(F(\PMC))$ and one side on $\{0\}\times S_-(F(\PMC))$, and one component of $R_+$ is a bigon with one side a component of $\Gamma$ and one side on $\{1\}\times S_-(F(\PMC))$. A cap is similar, except with two sides of the hexagon on $\{0\}\times S_-(F(\PMC))$ and one side on $\{1\}\times S_-(F(\PMC))$, and the bigon has boundary in $\{0\}\times S_-(F(\PMC))$.
These bordered-sutured invariants are easy to describe, given the invariants of the identity cobordism (See Section~\ref{sec:cup-cap-sutures}).
\item  An $R_-$-cup or cap. The Heegaard diagram is shown in Figure~\ref{fig:cupInRplus}. This is similar to the previous case, but with the roles of $R_\pm$ exchanged. Again, these bordered-sutured invariants are easy to describe (See Section~\ref{sec:cup-cap-sutures}).
\item Capping off a pointless bordered arc. Given an arc diagram
  $\PMC=(\{Z_j\},\mathbf{a},M)$ so that one of the arcs, say $Z_0$, has
  $Z_0\cap\mathbf{a}=\emptyset$, let $\PMC'=(\{Z_j\mid j\neq 0\},\mathbf{a},M)$ be the arc diagram
  consisting of all the arcs except $Z_0$. There is a bordered-sutured
  cobordism from $\PMC$ to $\PMC'$ which is the disjoint union of the
  identity cobordism of $\PMC'$ and a $3$-ball with one suture. The
  corresponding Heegaard diagram is shown in Figure~\ref{fig:pointless}. The
  bordered-sutured invariant for this  cobordism is easy to describe
  (see Section~\ref{sec:pointless}).
  There is also a dual operation, creating a pointless bordered arc, but we will not need this operation.
\end{enumerate}

\begin{figure}
  \centering
  \includegraphics{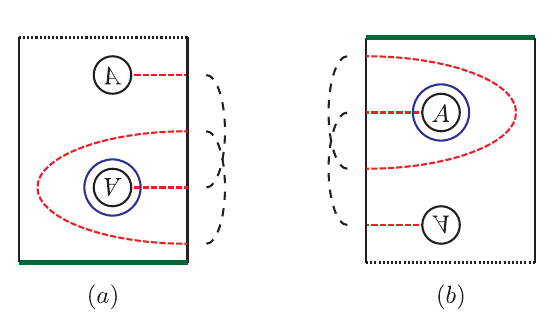}
  \caption{\textbf{Attaching a $1$- or $2$-handle to the interior of $Y$.} (a) Attaching a $1$-handle with both feet on the same boundary component. (b) Attaching a $2$-handle without disconnecting a boundary component. The $\beta$-circles are \textcolor{blue}{solid}, the $\alpha$-arcs are \textcolor{red}{dashed}, and the sutured boundary is \textcolor{darkgreen}{thick}. Beyond the dotted lines, each diagram is the same as the standard Heegaard diagram for the identity cobordism.}
  \label{fig:interiorHandle}
\end{figure}

\begin{figure}
  \centering
  \includegraphics{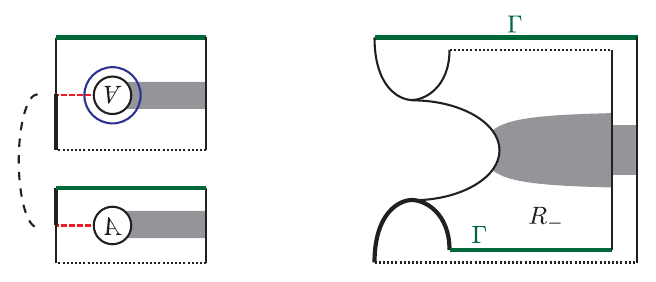}
  \caption{\textbf{Attaching a $1$-handle to $R_-$.} Left: part of a bordered-sutured Heegaard diagram corresponding to attaching a $1$-handle to $R_-$. The $\beta$-circles are \textcolor{blue}{solid}, the $\alpha$-arcs are \textcolor{red}{dashed}, and the sutured boundary is the \textcolor{darkgreen}{thick} horizontal lines.  Right: an illustration of attaching a $1$-handle to $R_-$. The shaded parts of the diagrams on the left and right correspond, as do the thick parts of the left boundaries.}
  \label{fig:RminusOneHandle}
\end{figure}

\begin{figure}
  \centering
  \includegraphics{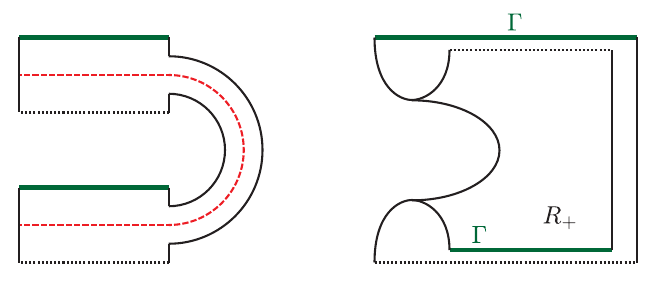}
  \caption{\textbf{Attaching a $1$-handle to $R_+$.} Left: part of a bordered-sutured Heegaard diagram corresponding to attaching a $1$-handle to $R_+$. Right: the corresponding $1$-handle being attached to $R_+$.}
  \label{fig:RplusOneHandle}
\end{figure}

\begin{figure}
  \centering
  \includegraphics{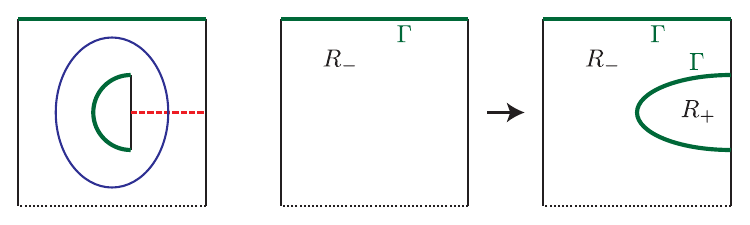}
  \caption{\textbf{A cup in $R_-$, i.e., an $R_+$-cup.} Left: the
    Heegaard diagram. Note that this piece of the diagram has two boundary components. Right: the effect on part of $R_-$. The rest of $R_\pm$ is unchanged. For a cap instead of a cup, rotate the diagram by $\pi$.}
  \label{fig:cupInRminus}
\end{figure}

\begin{figure}
  \centering
  \includegraphics{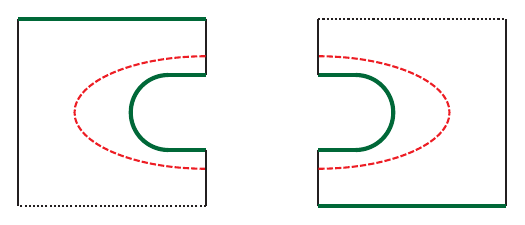}
  \caption{\textbf{A cup or cap in $R_+$, i.e., an $R_-$-cup.} Left: the Heegaard diagram for a cup. Right: the Heegaard diagram for a cap. The effect of the cup on $R_\pm$ is similar to Figure~\ref{fig:cupInRminus}, except with the roles of $R_\pm$ exchanged.}
  \label{fig:cupInRplus}
\end{figure}

\begin{figure}
  \centering
  \includegraphics{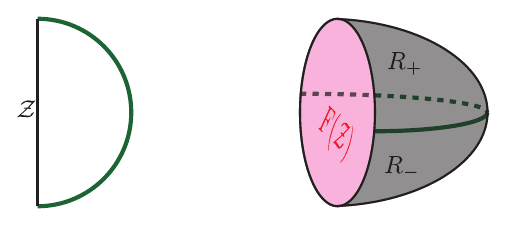}
  \caption{\textbf{Capping off a pointless bordered arc.} Left: the
    relevant component of the bordered-sutured Heegaard diagram. The
    bordered boundary is on the left, and thin, and the sutured
    boundary is on the right, and \textcolor{darkgreen}{thick}. The rest of the Heegaard
    diagram is the standard Heegaard diagram fro the identity map of
    $\PMC'$. Right: the non-identity component of the corresponding
    bordered-sutured cobordism. The bordered boundary is
    \textcolor{pink}{lightly} shaded and $R_\pm$ are darkly shaded.}
  \label{fig:pointless}
\end{figure}

Our first milestone is to prove that any bordered-sutured manifold can
be decomposed into these basic pieces. We start by showing that
arcslides generate the mapping class groupoid in the bordered-sutured
case. (In the connected boundary case, this is
well-known~\cite{Penner87:MCGroupoid,Bene08:ChordDiagrams,LOT4}.)
\begin{definition}
  The \emph{mapping class groupoid} has objects the arc diagrams $\PMC$. The morphism set $\Hom(\PMC_1,\PMC_2)$ is the set of isotopy classes of diffeomorphisms $F(\PMC_1)\to F(\PMC_2)$ taking the arcs $S_+\subset\bdy F(\PMC_1)$ homeomorphically to the arcs $S_+\subset F(\PMC_2)$.
\end{definition}
The mapping class groupoid is a disjoint union of groupoids, one for each topological type of sutured surface.

Given an arcslide from $\PMC_1$ to $\PMC_2$ there is a corresponding \emph{arcslide diffeomorphism} from $F(\PMC_1)$ to $F(\PMC_2)$ (see, e.g.,~\cite[Figure 3]{LOT4}). Abusing terminology, we will typically refer to the arcslide diffeomorphism as an arcslide.

\begin{lemma}
  The arcslide diffeomorphisms generate the mapping class groupoid.
\end{lemma}
\begin{proof}
  Fix arc diagrams $\PMC_1$ and $\PMC_2$ and a diffeomorphism $\phi\co
  F(\PMC_1)\to F(\PMC_2)$ respecting the markings of the
  boundary. Choose Morse functions $f_i$ on $F(\PMC_i)$ compatible
  with the arc diagrams (\emph{$\PMC$-compatible Morse functions} in
  Zarev's language~\cite[Definition 2.3]{Zarev09:BorSut}), and so that
  there is a neighborhood $U$ of $\bdy F(\PMC_2)$ so that
  $f_1|_{\phi^{-1}(U)}=f_2\circ\phi|_{\phi^{-1}(U)}$. The functions
  $f_1$ and $f_2\circ\phi$ can be connected by a $1$-parameter family
  of Morse functions $f_t$, $t\in[0,1]$, all of which agree with $f_1$
  over $\phi^{-1}(U)$.
  (See, e.g.,~\cite{Sharko98:MorseOnSurf}.) For a generic choice of metric, there are finitely many $t$ for which  $f_t$ is not Morse-Smale, because of a flow between two index $1$ critical points. These are the arcslides.
\end{proof}

A second lemma allows us to restrict to bordered-sutured manifolds
with connected boundary. Given a bordered-sutured manifold
$\mathcal{Y}=(Y,\Gamma,\phi)$, a \emph{decomposing disk} is a disk $D$ in $Y$ with
boundary in the sutured boundary of $Y$, and so that $\bdy D\cap R_+$
and $\bdy D\cap R_-$ each consists of one arc. Given a decomposing
disk $D$, $Y\setminus \nbd(D)$ can be made into a bordered-sutured
manifold by including $\bdy\nbd(D)$ in the sutured boundary, and
adding a single arc in each of the components of $\bdy\nbd(D)$ to
$\Gamma$. The bordered boundary of $Y\setminus\nbd(D)$ inherits a
parametrization from $\mathcal{Y}$. We say that such bordered-sutured manifolds differ by a \emph{disk decomposition}.

\begin{lemma}\label{lem:disk-decomp}
  If $\mathcal{Y}$ and $\mathcal{Y}'$ are bordered-sutured manifolds which
  differ by a disk decomposition then $\BSD(\mathcal{Y})\simeq\BSD(\mathcal{Y}')$.
\end{lemma}
\begin{proof}
  This is a special case of Zarev's surface decomposition
  theorem~\cite[Theorem 10.6]{Zarev09:BorSut}, but is also
  easy to see directly. We can find a bordered-sutured
  Heegaard diagram $(\Sigma,\alphas,\betas)$ for $\mathcal{Y}$ so that
  $D\cap \Sigma$ is a single arc $\eta$ with $\bdy\eta$ contained
  in the sutures of $\Sigma$ and so that $\eta$ is disjoint from
  $\alphas$ and $\betas$. Then
  $(\Sigma\setminus\nbd(\eta),\alphas,\betas)$ is a bordered-sutured
  Heegaard diagram for $\mathcal{Y}'$, where we include $\bdy\nbd(\eta)$ as
  part of the sutured boundary of $\Sigma\setminus\nbd(\eta)$. If we choose corresponding
  almost complex structures, the generators and differential for
  $\BSD(\Sigma,\alphas,\betas)$ and
  $\BSD(\Sigma\setminus\nbd(\eta),\alphas,\betas)$ are exactly the
  same.
\end{proof}

\begin{figure}
  \centering
  \includegraphics{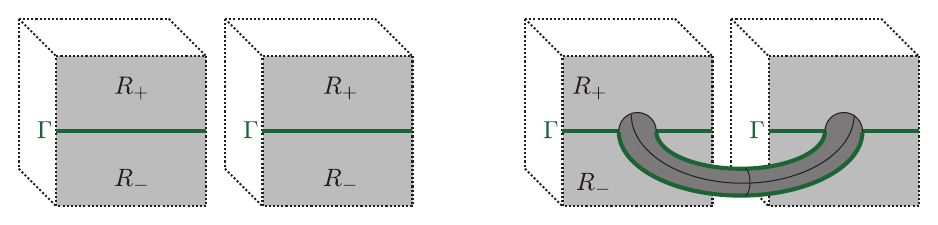}
  \caption{\textbf{Attaching a 3-dimensional $1$-handle to $\Gamma$.}
    Left: part of a bordered-sutured cobordism $Y$. The part of
    $\bdy Y$ shown is lightly shaded. Outside the dashed cubes, $Y$ is
    arbitrary. Right: the result of attaching a $1$-handle to connect
    these two components of $\bdy Y$.}
  \label{fig:connect-bdy-3d}
\end{figure}

\begin{figure}
  \centering
  \includegraphics{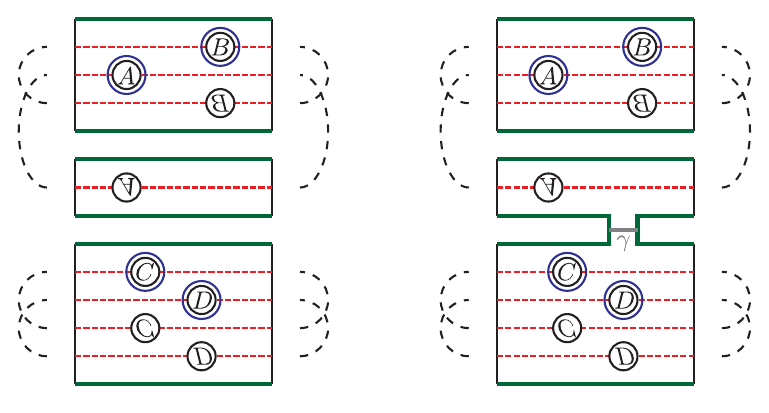} 
  \caption{\textbf{Making the boundary connected.} Left: a bordered-sutured Heegaard diagram for a bordered-sutured manifold with two boundary components. Right: a bordered-sutured Heegaard diagram for a bordered-sutured manifold with connected boundary, and an arc $\eta$, so that a disk decomposition along $\eta$ gives the diagram on the left. }
  \label{fig:connect-bdy}
\end{figure}

In particular, given any bordered-sutured manifold $\mathcal{Y}=(Y,\Gamma,\phi)$ one can attach 
$3$-dimensional $1$-handles to $\bdy Y$ with attaching points on $\Gamma$ to make $\bdy Y$ connected,
without changing $\BSD(\mathcal{Y})$. In this operation, each attaching disk intersects $\Gamma$ in an arc, and $\Gamma$ changes by surgery  at the attaching points so that these arcs are replaced with two parallel arcs that go over the $1$-handle and the result is a bordered-sutured manifold. See Figure~\ref{fig:connect-bdy-3d} for an
illustration of this operation and Figure~\ref{fig:connect-bdy} for the
corresponding operation on Heegaard diagrams.

Thus, it suffices to
compute $\BSD(\mathcal{Y})$ where $\mathcal{Y}$ has connected boundary.
\begin{lemma}\label{lem:decompose-bs}
  Up to disk decomposition, any bordered-sutured manifold
  $\mathcal{Y}=(Y,\Gamma,\phi\co F(\PMC)\to \bdy Y)$ can be decomposed as a union of the seven basic
  bordered-sutured pieces above.
\end{lemma}
\begin{proof}  
  Since attaching a $1$-handle to $\Gamma$ is the inverse of a disk
  decomposition, we may assume that $\bdy Y$ is connected and that the
  sutured part of $\bdy Y$ is connected.

  Choose some parametrization $\psi\co F(\PMC')\to\bdy Y$ of $\bdy Y$
  by the surface associated to a pointed matched circle $\PMC'$.  Then we can view
  $\mathcal{Y}$ as the composition of the bordered manifold $(Y,\psi)$
  and a bordered-sutured cobordism $(Y',\Gamma')$ from $F(\PMC')$ to
  $F(\PMC)$ so that $Y'$ is topologically a cylinder
  $[0,1]\times F(\PMC')$.

  From the bordered case~\cite[Section 8]{LOT4}, we can decompose
  $(Y,\psi)$ as a composition of arcslides and $1$- and $2$-handle
  attachments to the bordered boundary. It remains to decompose $Y'$
  into $1$- and $2$-handle attachments to $R_\pm$, $R_\pm$ cups and
  caps, capping pointless arcs, and arcslides. To this end, let $\bdy_sY'\subset \bdy Y'$ denote the sutured
  boundary. Choose a Morse function $f\co \bdy_sY'\to[0,1]$ so that
  $f^{-1}(0)=\bdy_sY'\cap F(\PMC')$, $f^{-1}(1)=\bdy_sY'\cap F(\PMC)$, $f$ has
  no index $0$ critical points, and each critical level of
  $f$ has a single critical point.
  By a possibly large perturbation of the sutures, we can arrange that
  for each $t\in[0,1]$ and each connected component $C$ of
  $f^{-1}(t)$, $C\cap \Gamma\neq \emptyset$. Further perturbing the
  sutures slightly, we may assume that:
  \begin{enumerate}
  \item If $p_1,\dots,p_k$ are the critical points of $f$ then each
    $p_i$ is in the interior of either $R_+$ or $R_-$ (i.e., $p_i\not\in\Gamma$).
  \item The restriction of $f$ to $\Gamma$ is a Morse function, with
    critical points $q_1,\dots,q_\ell$, say.
  \item The real numbers $f(p_1),\dots,f(p_k),f(q_1),\dots,f(q_\ell)$
    are all distinct.
  \end{enumerate}
  Choose $0=t_0<t_1<\dots<t_{k+\ell}=1$ so that each
  $f^{-1}((t_i,t_{i+1}))$ contains exactly one $p_i$ or $q_i$. Then
  $f^{-1}([t_i,t_{i+1}])$ corresponds to a $1$-handle attachment to
  $R_+$ or $R_-$ (for a $p_i$) or an $R_+$- or $R_-$-cup or cap or
  capping off a pointless bordered arc (for a
  $q_i$). Let
  \[
    \mathcal{Y}_i=(Y_i,\ \phi_{i,L}\co -F(\PMC_{i,L})\to \bdy_L Y_i,\ \phi_{i,R}\co F(\PMC_{i,R})\to \bdy_R Y_i)
  \]
  be the corresponding bordered-sutured cobordism (parametrized as in
  Figures~\ref{fig:RminusOneHandle}--\ref{fig:pointless}, above). 
  For
  $i=1,\dots,k+\ell-2$, let $\mathcal{Y}_{i,i+1}$ be the mapping cylinder of a
  composition of arclides from $\phi_{i,R}$ to $\phi_{i+1,L}$. Let
  $\mathcal{Y}_{k+\ell-1,k+\ell}$ be the mapping cylinder of a sequence of arcslides from $\phi_{k+\ell}$ to $\phi$. Then,
  \[
    \mathcal{Y}=(Y,\psi)\cup \mathcal{Y}_1\cup \mathcal{Y}_{1,2}\cup \mathcal{Y}_2\cup\cdots\cup \mathcal{Y}_{k+\ell-1,k+\ell}\cup \mathcal{Y}_{k+\ell-1}\cup \mathcal{Y}_{k+\ell-1,k+\ell},
  \]
  as bordered-sutured manifolds.  This completes the proof.
\end{proof}

Our next task is to compute the bordered-sutured modules associated to
the seven basic bordered-sutured pieces.  We start with some tools for
deducing bordered-sutured computations from bordered computations.
\begin{definition}
  Let $\PMC=(Z,\mathbf{a},M)$ and $\PMC'=(Z',\mathbf{a}',M')$ be
  arc diagrams.
  We say that
  $\PMC$ is a \emph{subdiagram} of $\PMC'$ if there is an embedding
  $\phi\co Z\to Z'$ so that $\phi(\mathbf{a})\subset \mathbf{a}'$ and
  $M$ is induced from $M'$. Note that we do not require that $\phi$ be
  proper, i.e., send the boundary of $Z$ to the boundary of $Z'$,
  or that $\phi^{-1}(\mathbf{a}')=\mathbf{a}$. If
  $\phi(\mathbf{a})=\mathbf{a}'$ (so, in particular,
  $\phi^{-1}(\mathbf{a}')=\mathbf{a}$) then we say that $\PMC$ is a
  \emph{full subdiagram} of $\PMC'$.

  Similarly, let $\HD=(\Sigma,\alphas,\betas,\z)$ and
  $\HD'=(\Sigma',\alphas',\betas',\z')$ be bordered-sutured Heegaard
  diagrams. We say that $\HD$ is a \emph{subdiagram} of $\HD'$ if
  there is an embedding $\phi\co \Sigma\to \Sigma'$ with the following
  properties:
  \begin{enumerate}
  \item $\phi$ sends the bordered boundary of $\Sigma$ to the bordered
    boundary of $\Sigma'$.
  \item $\phi(\alphas)\subset\alphas'$.
  \item $\phi(\betas)=\betas'$ and $\phi(\alphas)$ contains every
    circle in $\alphas'$.
  \end{enumerate}
  If further $\phi(\alphas)=\alphas'$ and then we say that $\HD$
  is a \emph{full subdiagram} of $\HD'$.
\end{definition}
In other words, a full subdiagram is obtained by turning some of the
bordered boundary of $\Sigma$ into sutured boundary, and in a
non-full subdiagram one can also forget some $\alpha$-arcs. 
See Figure~\ref{fig:subdiag} for some examples of subdiagrams.

\begin{figure}
  \centering
  \includegraphics{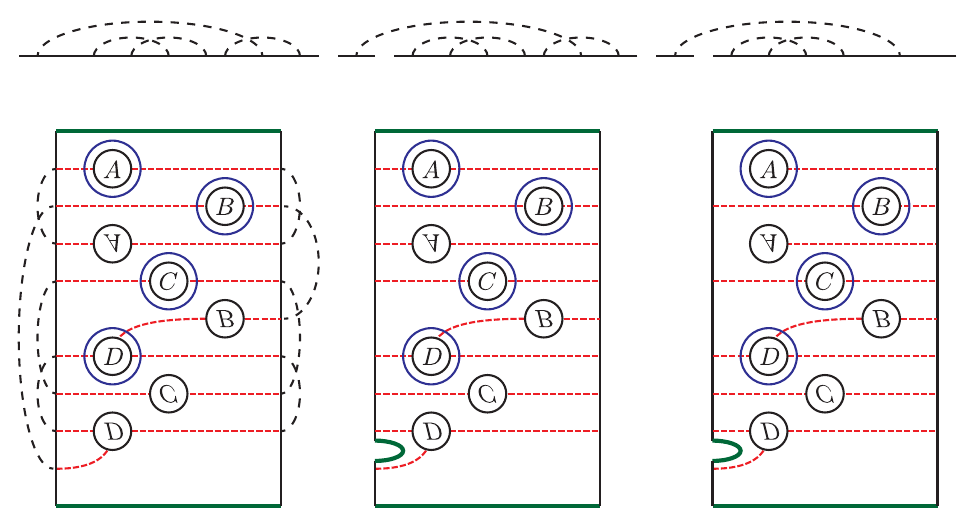}
  \caption{\textbf{Examples of subdiagrams of arc diagrams and Heegaard diagrams.} Top: a pointed matched circle $\PMC$, a full subdiagram of $\PMC$, and a non-full subdiagram of $\PMC$. Bottom: the standard Heegaard diagram for an arcslide, a full subdiagram of this Heegaard diagram, and a non-full subdiagram of this Heegaard diagram.}
  \label{fig:subdiag}
\end{figure}

If $\PMC$ is a subdiagram of $\PMC'$ then there is an injective
homomorphism $i\co\Alg(\PMC)\to \Alg(\PMC')$ obtained by regarding a
strand diagram in $\PMC$ as lying in $\PMC'$. If $\PMC$ is a full
subdiagram of $\PMC'$ then there is also a projection map
$p\co\Alg(\PMC')\to\Alg(\PMC)$ which is the identity on strand
diagrams contained entirely in $\PMC$ and sends any strand diagram not
entirely contained in $\PMC$ to $0$. Associated to $i$ is a
\emph{restriction of scalars} functor
\[
  i^*\co \ModCat_{\Alg(\PMC')}\to\ModCat_{\Alg(\PMC)}
\]
and associated to $p$ is an \emph{extension of scalars functor}
\begin{align*}
  p_*\co \lsup{\Alg(\PMC')}\ModCat&\to \lsup{\Alg(\PMC)}\ModCat\\
  \lsup{\Alg(\PMC')}P&\mapsto \lsup{\Alg(\PMC)}[p]_{\Alg(\PMC')}\DT\lsup{\Alg(\PMC')}P,
\end{align*}
where $\lsup{\Alg(\PMC)}[p]_{\Alg(\PMC')}$ denotes the rank $1$ type
\DA\ bimodule associated to $p$~\cite[Definition 2.2.48]{LOT2}.
\begin{lemma}\label{lem:induct-or-restrict}
  If $\HD$ is a subdiagram of $\HD'$ then
  $\BSA(\HD)\cong i^*\BSA(\HD')$. If $\HD$ is a full subdiagram of
  $\HD'$ then $\BSD(\HD)\cong p_*\BSD(\HD')$.
\end{lemma}
(Compare~\cite[Sections 6.1 and 6.2]{LOT2}.)
\begin{proof}
  This is immediate from the definitions.
\end{proof}

\begin{figure}
  \centering
  \includegraphics{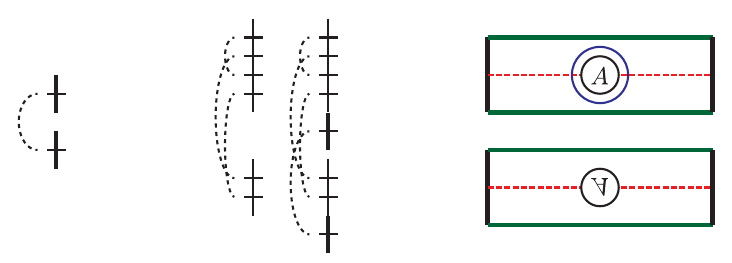}
  \caption{\textbf{The diagram $\PMC_b$ and its use.} Left: the arc
    diagram $\PMC_b$. Center: an arc diagram which is not a full
    subdiagram of a pointed matched circle and the result of gluing
    this diagram to $\PMC_b$, to give an arc diagram which is a full
    subdiagram of a pointed matched circle. Right: the standard
    Heegaard diagram for the identity cobordism of $\PMC_b$.}
  \label{fig:PMCb}
\end{figure}

Let $\PMC_b$ be the arc diagram with two intervals so that
$\mathbf{a}$ has a single point on each. See Figure~\ref{fig:PMCb}.

\begin{lemma}\label{lem:arc-sub-pmc}
  Let $\PMC$ be an arc diagram. Then there is a pointed matched circle
  $\overline{\PMC}$ so that either $\PMC$ or $\PMC\amalg \PMC_b$ is a
  full subdiagram of $\overline{\PMC}$.
\end{lemma}
\begin{proof}
  Write $\PMC=(Z,\mathbf{a},M)$.  There is a pairing of the endpoints
  of $Z$ by saying that $z$ is paired with $z'$ if, after doing
  surgery on $Z$ according to $(\mathbf{a},M)$ the points $z$ and $z'$
  are on the same connected component.  If $Z$ has more than two
  components then we can choose a top endpoint $z$ of an interval $I$
  and a bottom endpoint $z'$ of a different interval $I'$ so that $z$
  is not paired to $z'$ and glue $z$ to $z'$. If $Z$ has two
  components then either we can choose such a pair of points or we can
  do so after gluing on a copy of $\PMC_b$ as in
  Figure~\ref{fig:PMCb}. This gluing gives an arc diagram with
  one fewer intervals.
\end{proof}

\begin{lemma}\label{lem:disconnected}
  If $\HD=\HD_1\amalg\HD_2$ is a disconnected bordered-sutured Heegaard diagram then
  \[
    \BSD(\HD)\cong\BSD(\HD_1)\otimes_{\Ftwo}\BSD(\HD_2).
  \]  
\end{lemma}
\begin{proof}
  This is immediate from the definitions.
\end{proof}

There is another operation which has essentially no effect on the bordered-sutured invariants. Suppose that $\HD$ is a bordered-sutured Heegaard diagram and $\eta$ is an arc in $\HD$ starting on the sutured boundary of $\HD$ and ending on the bordered boundary, and so that $\eta$ is disjoint from the $\alpha$- and $\beta$-curves. Cutting along $\eta$ gives a new bordered-sutured Heegaard diagram $\HD'$. We will say that $\HD'$ is obtained from $\HD$ by a \emph{safe cut}. The modules $\BSD(\HD)$ and $\BSD(\HD')$ are essentially the same, the subtlety being that they are over different algebras.
\begin{lemma}\label{lem:cut-HD-no-effect}
  Suppose that $\HD'$ is obtained from $\HD$ by a safe cut. Let $\PMC$
  (respectively $\PMC'$) be the arc diagram on the boundary of $\HD$
  (respectively $\HD'$). Then there is an inclusion
  $i\co \Alg(\PMC')\into\Alg(\PMC)$ (induced by the inclusion of
  $\PMC'$ into $\PMC$) and a projection
  $p\co \Alg(\PMC)\to\Alg(\PMC')$ (induced by setting any strand
  diagram not contained in $\PMC'$ to zero) so that
  \begin{align}
    \BSD(\HD)&\simeq i_*\BSD(\HD')\label{eq:iBSD}\\
    \BSD(\HD')&\simeq p_*\BSD(\HD)\label{eq:pBSD} \\
    \BSA(\HD)&\simeq p^*\BSA(\HD')\\
    \BSA(\HD')&\simeq i^*\BSA(\HD).
  \end{align}
  (Here, $i_*$ and $p_*$ are induction, or extension of scalars and
  $i^*$ and $p^*$ are restriction of scalars.)  
\end{lemma}
\begin{proof}
  Write $\HD=(\Sigma,\alphas,\betas,\z)$ and $\HD'=(\Sigma',\alphas',\betas',\z')$. There is a map $\iota\co \Sigma'\to\Sigma$ (which is $2$-to-$1$ on the new sutured part of the boundary of $\Sigma'$ and $1$-to-$1$ otherwise). Abusing notation, we will also let $\iota$ denote the induced map $\Sigma'\times[0,1]\times\RR\to \Sigma\times[0,1]\times\RR$.

  The map $\iota$ induces a bijection of generators, $\x'=\{x'_i\}\mapsto \iota(\x')=\{\iota(x'_i)\}$, and a bijection of domains $\iota_*\co \pi_2(\x',\y')\to \pi_2(\iota(\x'),\iota(\y'))$. Choose a sufficiently generic almost complex structure $J$ on $\Sigma\times[0,1]\times\RR$~\cite[Definition 5.7]{LOT1} and let $J'=\iota^*J$ be the induced almost complex structure on $\Sigma'\times[0,1]\times\RR$. Then for $B'\in\pi_2(\x',\y')$, a map $u\co S\to \Sigma'\times[0,1]\times\RR$ is a $J'$-holomorphic curve in the moduli space $\cM^{B'}(\x',\y')$ if and only if $\iota\circ u$ is a $J$-holomorphic curve in $\cM^{\iota_*(B')}(\iota(\x'),\iota(\y'))$. The assignment $u\mapsto \iota\circ u$ is clearly injective; from positivity of domains~\cite[Proof of Lemma 5.4]{LOT1}, it is also surjective.

  Thus, if $a\otimes \y$ is a term in $\delta^1(\x)$ for some
  generators $\x,\y\in\BSD(\HD)$ and strand diagram $a\in\Alg(\PMC)$
  then $a$ is the image of a strand diagram
  $a'\in\Alg(\PMC')$. Further, $a'\otimes \y'$ occurs in the
  differential of $\x'$, where $\iota(\x')=\x$ and
  $\iota(\y')=\y$. The results about $\BSD$ follow. The proofs for
  $\BSA$ are only notationally different.
\end{proof}

\begin{lemma}\label{lem:BSD-id-PMCb}
  The module $\BSD(\Id_{\PMC_b})$ associated to the identity cobordism
  of $\PMC_b$ has two generators and trivial differential.
\end{lemma}
\begin{proof}
  Again, this is immediate from the definitions. (See also
  Figure~\ref{fig:PMCb}.)
\end{proof}

\subsubsection{The identity cobordism}\label{sec:identity}
Fix an arc diagram $\PMC=(Z,\mathbf{a},M)$. 
Given a set of matched pairs $\mathbf{s}$ in $\PMC$, let $\mathbf{s}^c$ denote the complementary set of matched pairs, but viewed as lying in $-\PMC$. We call the pair of idempotents $I(\mathbf{s})\otimes I(\mathbf{s}^c)\in \Alg(\PMC)\otimes\Alg(-\PMC)$ \emph{complementary}.

A \emph{chord} in $\PMC$ is an interval in $Z$ with endpoints in $\mathbf{a}$. Given a chord $\rho$ in $\PMC$ there is a corresponding algebra element $a(\rho)$, which is the sum of all strand diagrams in which $\rho$ is the only moving strand. There is also a corresponding chord $\rho'$ in $-\PMC$.

\begin{lemma}\label{prop:DD-id-is}
  For any arc diagram $\PMC$, the bordered-sutured bimodules $\BSDD(\Id_{\PMC})$ associated to the identity cobordism of $\PMC$ is generated by the set of pairs of complementary idempotents $(I\otimes J)\in\Alg(\PMC)\otimes\Alg(-\PMC)$, and has differential
  \[
    \delta^1(I\otimes J)=\sum_{I'\otimes J'}\sum_{\text{chords }\rho} (I\otimes J)(a(\rho)\otimes a(\rho'))\otimes (I'\otimes J'),
  \]
  where the first sum is over the pairs of complementary idempotents.
\end{lemma}
\begin{proof}
  We deduce this from the corresponding statement for pointed matched
  circles~\cite[Theorem 1]{LOT4}. Provisionally, let $\DDid(\PMC)$ be
  the type \DD\ bimodule described in the statement of the lemma, and
  let $\BSDD(\Id_{\PMC})$ be the type \DD\ bimodule associated to the
  standard Heegaard diagram for the identity cobordism of $\PMC$ (with
  respect to any sufficiently generic almost complex structure). We
  want to show that $\DDid(\PMC)\simeq \BSDD(\Id_{\PMC})$. Note that
  there is a unique isomorphism of underlying
  $\Idem(\PMC)\otimes\Idem(-\PMC)$-modules between $\DDid(\PMC)$ and
  $\BSDD(\Id_{\PMC})$; we will show that this identification
  intertwines the operations $\delta^1$.

  By definition, $\DDid(\PMC\amalg \PMC_b)\cong
  \DDid(\PMC)\otimes_{\Ftwo}\DDid(\PMC_b)$. Similarly, by
  Lemma~\ref{lem:disconnected}, $\BSDD(\Id_{\PMC\amalg\PMC_b})\cong
  \BSDD(\Id_{\PMC})\otimes_{\Ftwo}\BSDD(\Id_{\PMC_b})$.  
  If the unique isomorphism of underlying modules $\DDid(\PMC\amalg
  \PMC_b)\cong \BSDD(\Id_{\PMC\amalg\PMC_b})$ intertwines the
  operations $\delta^1$ then so does the unique isomorphism 
  $\DDid(\PMC)\cong \BSDD(\Id_{\PMC})$.
So, by Lemma~\ref{lem:arc-sub-pmc}, we may assume that $\PMC$ is a full subdiagram of a pointed matched circle $\overline{\PMC}$. By Lemma~\ref{lem:induct-or-restrict}, $\BSDD(\Id_{\PMC})\cong p_*\BSDD(\Id_{\overline{\PMC}})$. Further, by inspection, $\DDid(\PMC)\cong p_*\DDid(\overline{\PMC})$. Finally, from the bordered case~\cite[Theorem 1]{LOT4}, $\BSDD(\Id_{\overline{\PMC}})\cong \DDid(\overline{\PMC})$. The result follows.
\end{proof}

We note that we have also computed $\BSAA(\Id_{\PMC})$, by an analogue of a duality result in the bordered case~\cite[Theorem 5]{LOTHomPair}:
\begin{proposition}\label{prop:compute-AA-id}
  For any arc diagram $\PMC$ there is a homotopy equivalence of
  $\Ainf$-bimodules
  \begin{multline*}
    \BSAA(\Id_{\PMC})_{\Alg(\PMC),\Alg(-\PMC)}\\
    \simeq
    \Mor_{\Alg(\PMC)}\bigl((\Alg(-\PMC)\otimes_{\Ftwo}\Alg(\PMC))_{\Alg(-\PMC)\otimes\Alg(\PMC)}\DT\lsup{\Alg(-\PMC)\otimes\Alg(\PMC)}\BSDD(\Id_{\PMC}),\
    \Alg(\PMC)\bigr)
  \end{multline*}
\end{proposition}
\begin{proof}
  The proof is the same as in the bordered case~\cite[Theorem 5]{LOTHomPair}.
\end{proof}

\begin{corollary}\label{cor:compute-DA}
  If $\mathcal{Y}$ is a bordered-sutured cobordism from $\PMC_1$ to $\PMC_2$ then
  \[
    \lsup{\Alg(\PMC_1)}\BSDA(\mathcal{Y})_{\Alg(\PMC_2)}\simeq \Mor^{\Alg(-\PMC_2)}\bigl(
    \Alg(\PMC_2)\DT\BSDD(\Id_{\PMC_2}),\lsup{\Alg(\PMC_1),\Alg(-\PMC_2)}\BSDD(\mathcal{Y})\bigr).
  \]
  (Here, $\Mor$ is the chain complex of type $D$ structure morphisms.)
\end{corollary}
\begin{proof}
  This is immediate from Proposition~\ref{prop:compute-AA-id} and the
  fact that $\Alg(-\PMC_2)\DT\cdot$ is a quasi-equivalence of \dg categories from the category of type $D$ structures to the category of $\Ainf$-modules.
\end{proof}

\subsubsection{Arcslides}\label{sec:arcslides}
Let $\phi\co \PMC\to\PMC'$ be an arcslide and $\HD_\phi$ the standard
Heegaard diagram for $\phi$~\cite[Figure 16]{LOT4} (see also Figures~\ref{fig:subdiag} and~\ref{fig:btau-HD}). Let
$\overline{\PMC}$ be the pointed matched circle from
Lemma~\ref{lem:arc-sub-pmc} and
$\overline{\phi}\co\overline{\PMC}\to\overline{\PMC}'$ the
corresponding arcslide.

\begin{proposition}\label{prop:bs-arcslide}
  The bimodule $\BSDD(\HD_\phi)$ is generated by the
  \emph{near-complementary pairs of idempotents}~\cite[Definition
  1.5]{LOT4} in $\Alg(\PMC)\otimes\Alg(\PMC')$ and the differential is
  given by
  \[
    \delta^1(I\otimes J)=\sum_{I'\otimes J'}(I\otimes J)A\otimes (I'\otimes J')
  \]
  where $A$ is the sum of all near-chords in
  $(\overline{\PMC},\overline{\PMC}')$~\cite[Definition 4.19]{LOT4}
  which are contained in $(\PMC,\PMC')$ if $\overline{\phi}$ is an
  underslide, and is the sum of all near-chords in
  $(\overline{\PMC},\overline{\PMC}')$ determined by any basic
  choice~\cite[Definition 4.32]{LOT4} which are contained in
  $(\PMC,\PMC')$ if $\overline{\phi}$ is an overslide.
\end{proposition}
\begin{proof}
  The proof is similar to the proof of Proposition~\ref{prop:DD-id-is}.
  By Lemmas~\ref{lem:disconnected} and~\ref{lem:BSD-id-PMCb}, it
  suffices to consider the case that $\PMC$ (rather than
  $\PMC\amalg\PMC_b$) is a full subdiagram of a pointed matched circle.
  In this case, $\HD_\phi$ is a full subdiagram of the standard
  Heegaard diagram for $\overline{\phi}$.  So, this is immediate from
  Lemma~\ref{lem:induct-or-restrict} and the computation for pointed
  matched circles~\cite[Propositions 4.20 and 4.37]{LOT4}.
\end{proof}

\subsubsection{Interior 1- and 2-handle attachments}\label{sec:interior-handle}
In this section we will give the bimodule for a $2$-handle attachment;
the bimodule for a $1$-handle attachment is the same except for
exchanging which action corresponds to which boundary component.

Before giving the bimodule we need notation for the algebra $\Alg(T^2)$ associated to the torus (with two boundary sutures). This algebra is the path algebra with relations
\[
  \mathcenter{
  \begin{tikzpicture}
    \node at (0,0) (i0) {$\iota_0$};
    \node at (2,0) (i1) {$\iota_1$};
    \draw[->, bend left=45] (i0) to node[above]{$\rho_1$} (i1);
    \draw[->] (i1) to node[above]{$\rho_2$} (i0);
    \draw[->, bend right=45] (i0) to node[above]{$\rho_3$} (i1);
  \end{tikzpicture}}
  /(\rho_2\rho_1=\rho_3\rho_2=0).
\]
(Compare~\cite[Section 11.1]{LOT1}. We are following the convention that $\rho_1\rho_2$ means the arrow labeled $\rho_1$ followed by the arrow labeled $\rho_2$; this is the opposite of composition order. So, for instance, $\rho_1=\iota_0\rho_1\iota_1$.)

Let $\HB_\infty$ be the $\infty$-framed solid torus, with boundary the genus $1$
pointed matched circle. The type $D$ structure $\CFDa(\HB_\infty)$ over
$\Alg(T^2)$ has a single generator $r$, and differential
\[
  \delta^1(r)=(\rho_2\rho_3)\otimes r.
\]
\cite[Section 11.2]{LOT2}.
We can regard $\CFDa(\HB_\infty)$ as a type \DD\ bimodule over $\Alg(T^2)$ and $\Ftwo$.

Given an arc diagram $\PMC$, let $\Id_{\PMC}$ denote the identity cobordism of $F(\PMC)$, and $\BSDD(\Id_{\PMC})$ the corresponding type \DD\ bimodule.

Consider a (non-disconnecting) $2$-handle attachment as in Figure~\ref{fig:interiorHandle}(b) where the pointed matched circle on the right is $\PMC_R$ and the pointed matched circle on the left is $\PMC_L$. Then there is an inclusion map
\begin{equation}\label{eq:tens-T2}
  \Alg(T^2)\otimes_{\Ftwo}\Alg(\PMC_R)\into\Alg(\PMC_L).
\end{equation}
\begin{proposition}\label{prop:interior-handle}
  The type \DD\ bimodule for a $2$-handle attachment from $\Alg(\PMC_L)$ to $\Alg(\PMC_R)$ is the image of $\BSDD(\Id_{\PMC_R})\otimes_{\Ftwo}\CFDa(\HB_\infty)$ under extension of scalars with respect to the map~\eqref{eq:tens-T2}. 
\end{proposition}
\begin{proof}
  This is immediate from Lemmas~\ref{lem:cut-HD-no-effect} and~\ref{lem:disconnected}.
\end{proof}

\subsubsection{Attaching handles to $R_\pm$}\label{sec:R-handles}
We describe the bimodules for attaching a $1$-handle to $R_-$ or $R_+$.

Let $\PMC'=(Z',\mathbf{a}',M')$ be an arc diagram, and let $z_1,z_2$ be terminal endpoints of two different components of $Z'$. Let $\PMC=(Z',\mathbf{a},M)$ where $\mathbf{a}=\mathbf{a}'\cup\{b,c\}$ and $b$ is slightly below $z_1$ and $c$ is slightly below $z_2$; and $M$ agrees with $M'$ on $\mathbf{a}'$ and pairs $b$ and $c$. That is, $-\PMC$ (respectively $\PMC'$) is as on the left (respectively right) of the Heegaard diagram in Figure~\ref{fig:RminusOneHandle}.
Given a strand diagram $a\in \Alg(\PMC')$, composing $a$ with the inclusion map $i\co \PMC'\into \PMC$ gives a strand diagram $i(a)\in\Alg(\PMC)$, and this induces an algebra homomorphism $\inobc\co \Alg(\PMC')\into\Alg(\PMC)$. There is an induced homomorphism
\[
  (\inobc\otimes\Id)\co \Alg(\PMC')\otimes\Alg(-\PMC')\into \Alg(\PMC)\otimes\Alg(-\PMC').
\]
\begin{proposition}\label{prop:R-minus-one-handle}
  If $\mathcal{Y}$ is the bordered-sutured cobordism from $\Alg(\PMC)$ to $\Alg(\PMC')$ associated to a $1$-handle attachment to $R_-$ (Figure~\ref{fig:RminusOneHandle}) then there is a homotopy equivalence
  \[
    \lsup{\Alg(\PMC),\Alg(-\PMC')}\BSDD(\mathcal{Y})\simeq (\inobc\otimes\Id)_*\lsup{\Alg(\PMC'),\Alg(-\PMC')}\BSDD(\Id_{\PMC'}).
  \]
\end{proposition}
\begin{proof}
  This is immediate from Lemmas~\ref{lem:disconnected} and~\ref{lem:cut-HD-no-effect} and the observation that there are no holomorphic curves with image in the non-identity part of the Heegaard diagram from Figure~\ref{fig:RminusOneHandle}.
\end{proof}

Next, let $\PMC=(Z,\mathbf{a},M)$ be an arc diagram. Suppose that $b$ and $c$ are points in $\mathbf{a}$ which are adjacent to terminal endpoints of arcs in $Z$ and so that $b$ is matched with $c$. Let $\PMC'=(Z',\mathbf{a}',M')$ be the result of performing surgery on the pair $(b,c)$, with $\mathbf{a}'=\mathbf{a}\setminus\{b,c\}$.
So, $-\PMC$ (respectively $\PMC'$) is the arc diagram on the left (respectively right) of the Heegaard diagram in Figure~\ref{fig:RplusOneHandle}.

There is an inclusion map $\ibc\co \Alg(\PMC')\into \Alg(\PMC)$, defined as follows. Recall that a strand diagram in $\PMC$ consists of subsets $S$, $T$ of the set of matched pairs in $\mathbf{a}$ and a collection of arcs $\rho$ with initial endpoints in $S$ and terminal endpoints in $T$, satisfying some conditions. 
Given a strand diagram $b=(S,T,\rho)$ in $\PMC'$, let 
\[
	\ibc(b)=(S\cup\{b,c\},T\cup\{b,c\},\rho).
\]
There is an induced homomorphism
\[
  (\ibc\otimes\Id)\co \Alg(\PMC')\otimes\Alg(-\PMC')\into \Alg(\PMC)\otimes\Alg(-\PMC').
\]

\begin{proposition}\label{prop:R-plus-one-handle}
  If $\mathcal{Y}$ is the bordered-sutured cobordism from $\Alg(\PMC)$ to $\Alg(\PMC')$ associated to a $1$-handle attachment to $R_+$ (Figure~\ref{fig:RplusOneHandle}) then there is a homotopy equivalence
  \[
    \lsup{\Alg(\PMC),\Alg(-\PMC')}\BSDD(\mathcal{Y})\simeq (\ibc\otimes\Id)_*\lsup{\Alg(\PMC'),\Alg(-\PMC')}\BSDD(\Id_{\PMC'}).
  \]
\end{proposition}
\begin{proof}
  Again, this is immediate from Lemma~\ref{lem:cut-HD-no-effect} and the observation that there are no holomorphic curves with image in the non-identity part of the Heegaard diagram from Figure~\ref{fig:RplusOneHandle}.
\end{proof}

\subsubsection{Cupping and capping sutures}\label{sec:cup-cap-sutures}
We describe the type \DD\ bimodules associated to introducing an $R_+$-cup (Figure~\ref{fig:cupInRminus}) and an $R_-$-cup (Figure~\ref{fig:cupInRplus}). The bimodules associated to capping sutures are the same as these, except with the actions reversed. These bimodules are also essentially the same as the bimodules from Section~\ref{sec:R-handles}.

Given an arc diagram $\PMC=(Z,\mathbf{a},M)$ and a terminal endpoint $z$ of a component of $Z$, let $\PMC'=(Z',\mathbf{a}',M')$ be the arc diagram where:
\begin{itemize}
\item $Z'=Z\amalg Z_0$, where $Z_0$ is a single interval.
\item $\mathbf{a'}=\mathbf{a}\cup \{b,c\}$ where $b$ is a point in $Z$ adjacent to $z$ and $c$ is a point in the interior of $Z_0$.
\item $M'$ agrees with $M$ on $\mathbf{a}$ and matches $b$ with $c$.
\end{itemize}
There is an inclusion map $\inobc\co \Alg(\PMC)\into \Alg(\PMC')$, which
simply views a strand diagram in $\PMC$ as lying in $\PMC'$. 
There is also an inclusion map $\ibc\co \Alg(\PMC)\into\Alg(\PMC')$ which sends a strand diagram $(S,T,\rho)$ to $(S\cup\{b,c\},(T\cup\{b,c\},\rho)$, i.e., places a pair of horizontal strands at $b$ and $c$ and otherwise leaves the strand diagram unchanged.

\begin{proposition}\label{prop:cup-in-R}
  Let $\mathcal{Y}_{\cupplus}$ (respectively $\mathcal{Y}_{\cupminus}$) be the bordered-sutured cobordism from $\PMC$ to $\PMC'$ which creates an $R_+$-cup (respectively $R_-$-cup) as in Figure~\ref{fig:cupInRminus} (respectively Figure~\ref{fig:cupInRplus}). Then there are homotopy equivalences
  \begin{align}
    \lsup{\Alg(\PMC),\Alg(-\PMC')}\BSDD(\mathcal{Y}_{\cupplus}) &\simeq(\Id\otimes \inobc)_*\lsup{\Alg(\PMC),\Alg(-\PMC)}\BSDD(\Id_{\PMC})\\
    \lsup{\Alg(\PMC),\Alg(-\PMC')}\BSDD(\mathcal{Y}_{\cupminus}) &\simeq (\Id\otimes\ibc)_*\lsup{\Alg(\PMC),\Alg(-\PMC)}\BSDD(\Id_{\PMC}).
  \end{align}
\end{proposition}
\begin{proof}
  This is the same as the proofs of Propositions~\ref{prop:R-minus-one-handle} and~\ref{prop:R-plus-one-handle}, and is left to the reader.
\end{proof}

\subsubsection{Capping off pointless arcs}\label{sec:pointless}
Suppose that $\PMC=(\{Z_j\},\mathbf{a},M)$ is an arc diagram so that
$Z_0\cap \mathbf{a}=\emptyset$ and that
$\PMC'=(\{Z_j\mid j\neq 0\},\mathbf{a},M)$. That is, $\PMC'$ is
obtained from $\PMC$ by deleting a pointless bordered arc. Then there
is a canonical isomorphism $\Alg(\PMC)\cong\Alg(\PMC')$.

\begin{proposition}\label{prop:pointless}
  If $\mathcal{Y}$ is a bordered-sutured cobordism from $\PMC$ to $\PMC'$
  corresponding to capping off a pointless bordered arc then
  $\BSDD(\mathcal{Y})\cong \BSDD(\Id_{\PMC})$ as a type \DD\ structure over
  $\Alg(\PMC)$ and $\Alg(-\PMC')\cong\Alg(-\PMC)$.
\end{proposition}
\begin{proof}
  This is immediate from Lemma~\ref{lem:disconnected}.
\end{proof}

\subsubsection{Putting it all together}
\begin{theorem}\label{thm:compute-borsut}
  For any bordered-sutured manifold $\mathcal{Y}$, the bordered-sutured modules
  $\BSD(\mathcal{Y})$ and $\BSA(\mathcal{Y})$ are algorithmically computable.
\end{theorem}
\begin{proof}[Proof by factoring]
  By Lemma~\ref{lem:decompose-bs}, we can decompose $\mathcal{Y}$ as a sequence
  of bordered-sutured cobordisms $\mathcal{Y}_n\circ\dots\circ \mathcal{Y}_1$ where each
  $\mathcal{Y}_i$ is one of the seven basic bordered-sutured pieces. The type
  \DD\ bimodule $\BSDD(\mathcal{Y}_i)$ is computed in
  Proposition~\ref{prop:bs-arcslide},~\ref{prop:interior-handle},~\ref{prop:R-minus-one-handle},~\ref{prop:R-plus-one-handle},~\ref{prop:cup-in-R},
  or~\ref{prop:pointless}. Corollary~\ref{cor:compute-DA} then
  computes $\BSDA(\mathcal{Y}_i)$, and Proposition~\ref{prop:compute-AA-id}
  computes $\BSAA(\mathcal{Y}_n)$, by tensoring $\BSDD(\mathcal{Y}_n)$ with a copy of $\BSAA(\Id)$ on each side. Then, we have
  \begin{align*}
    \BSD(\mathcal{Y})&\simeq \BSDA(\mathcal{Y}_n)\DT\BSDA(\mathcal{Y}_{n-1})\DT\cdots\DT\BSDA(\mathcal{Y}_2)\DT\BSD(\mathcal{Y}_1)\\
    \BSA(\mathcal{Y})&\simeq \BSAA(\mathcal{Y}_n)\DT\BSDA(\mathcal{Y}_{n-1})\DT\cdots\DT\BSDA(\mathcal{Y}_2)\DT\BSD(\mathcal{Y}_1).\qedhere
  \end{align*}
\end{proof}
We note that a less computationally useful proof of
Theorem~\ref{thm:compute-borsut}, via nice diagrams, already appears in Zarev's
work:
\begin{proof}[Alternative proof via nice diagrams]
  Any bordered-sutured Heegaard diagram can be made \emph{nice} by a
  sequence of Heegaard moves~\cite[Proposition 4.17]{Zarev09:BorSut},
  and the computation of $\BSD$ from a nice diagram is
  algorithmic~\cite[Theorem 7.14]{Zarev09:BorSut}.
\end{proof}

As an aside, the first proof of
Theorem~\ref{thm:compute-borsut} also gives an algorithm for computing
sutured Floer homology which is independent from, and presumably
faster than, the ``nice diagram'' approach~\cite[Theorems 6.4 and 7.4]{Juhasz08:SuturedDecomp}.

\subsubsection{The twisting bimodule}
\begin{figure}
  \centering
  \includegraphics{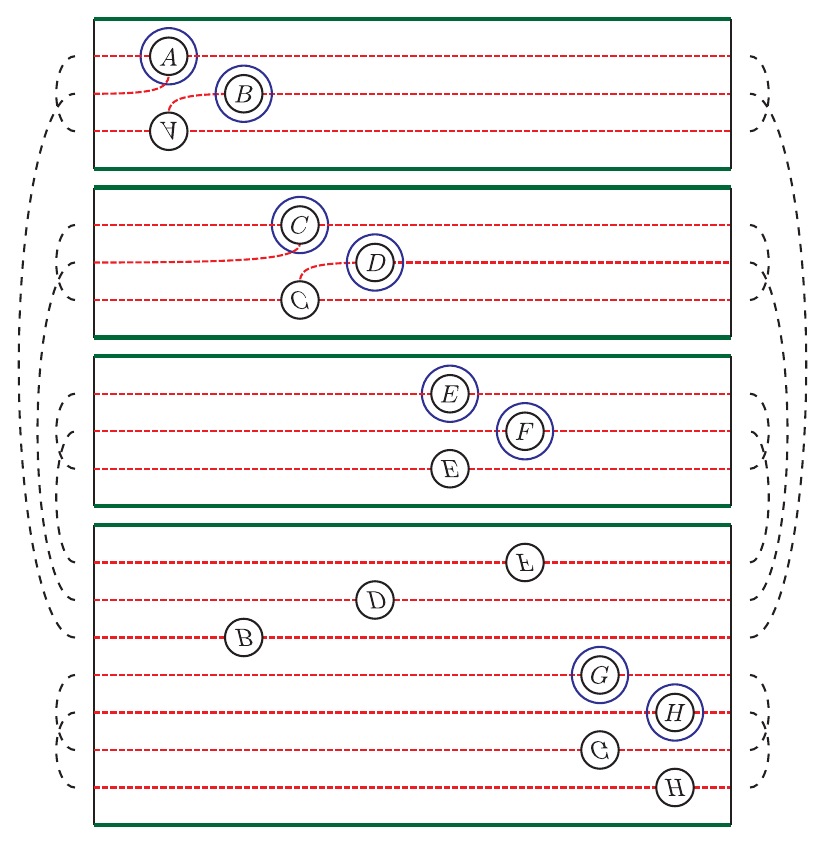}
  \caption{\textbf{Bordered-sutured Heegaard diagram for $\btau$.} The case shown has $\bdy Y$ of genus $1$ and $\ell=2$. Note that this diagram is a destabilization of the composition of the standard Heegaard diagrams for two arcslides.}
  \label{fig:btau-HD}
\end{figure}
Theorem~\ref{thm:compute-borsut} gives a way of computing the bimodule $\btau$, resolving difficulty~\ref{item:comp-1.5}. Explicitly, if the boundary of $Y$ has genus $k$ and $T$ consists of $\ell$ arcs and some number of circles then, for an appropriate choice of arc diagram we can factor $\btau$ as a product of $\ell$ arcslides; see Figure~\ref{fig:btau-HD}. So, the bimodule for $\btau$ is straightforward to compute from Proposition~\ref{prop:bs-arcslide} and Corollary~\ref{cor:compute-DA}.

\subsection{Evading the Novikov ring}
Next we turn to difficulty~\ref{item:comp-2}. The idea is to replace the Novikov field $\Lambda$ with the field of rational functions $\Ftwo(x_1,\dots,x_{4k-1})$, over which computing the homology of chain complexes is clearly algorithmic, and $\lsup{\Alg(\PMC)}\bLambda(\PMC)_{\Alg(\PMC)}$ by a type \DA\ bimodule $\lsup{\Alg(\PMC)}\bFrac(\PMC)_{\Alg(\PMC)}$. Before definition $\bFrac$, note that there is an embedding $\nu\co \ZZ^{4k-1}\into \Ftwo(x_1,\dots,x_{4k-1})^\times$ given by
\[
  \nu(a_1,\dots,a_{4k-1})=x_1^{a_1}\cdots x_{4k-1}^{a_{4k-1}}.
\]

\begin{definition}
  As an
  $\Ftwo(x_1,\dots,x_{4k-1})\otimes_{\Ftwo}\Idem(\PMC)$-module, define
  \[
    \bFrac(\PMC)=\Ftwo(x_1,\dots,x_{4k-1})\otimes_{\Ftwo}\Idem(\PMC).
  \]
  The structure maps
  $\delta^1_n$ on $\bFrac(\PMC)$ are defined to vanish if $n\neq 2$, and
  $\delta^1_2\co \bFrac(\PMC)\otimes \Alg(\PMC)\to \Alg(\PMC)\otimes
  \bFrac(\PMC)$ is a homomorphism of $\Ftwo(x_1,\dots,x_{4k-1})$-vector spaces. So, it
  only remains to define $\delta^1(i\otimes a)$ for $i$ a basic idempotent
  (viewed as a generator of $\bFrac(\PMC)$) and $a$ a
  strand diagram. Define
  \[
    \delta^1_2(i\otimes a)=
    \begin{cases}
      a\otimes \nu([a])j& \text{if }ia\neq 0,\text{ where }aj\neq 0\\
      0 & \text{otherwise}.
    \end{cases}
  \]
\end{definition}

\begin{theorem}
  Theorems~\ref{thm:detect-incompress} and~\ref{thm:detect-tang} hold with $\lsup{\Alg(\PMC)}\bLambda(\PMC)_{\Alg(\PMC)}$ replaced by $\lsup{\Alg(\PMC)}\bFrac(\PMC)_{\Alg(\PMC)}$.
\end{theorem}
\begin{proof}
  The map $\psi$ induces an injection $\Ftwo(x_1,\dots,x_{4k-1})\into \Lambda$, and
  \[
    \lsup{\Alg(\PMC)}\bLambda(\PMC)_{\Alg(\PMC)}\cong \lsup{\Alg(\PMC)}\bFrac(\PMC)_{\Alg(\PMC)}\otimes_{\Ftwo(x_1,\dots,x_{4k-1})}\Lambda.
  \]
  So, the result follows from the fact that, over a field, tensor
  product is an exact functor.
\end{proof}

A somewhat improved version of the bordered algorithm for computing $\CFDa(Y)$ has been implemented by Zhan~\cite{Zhan:bfh}. Although we have not done so, it should be relatively straightforward to extend his code to compute bordered-sutured Floer homology, check incompressibility, and check boundary parallelness.

\bibliographystyle{hamsalpha}
\bibliography{heegaardfloer}
\end{document}